\documentclass[11pt]{amsart}
\usepackage{amssymb}
\usepackage{hyperref}
\usepackage[a4paper]{geometry}
\usepackage{dsfont}
\usepackage{fullpage}
\usepackage{amsmath}
\theoremstyle{plain}
\newtheorem{thm}{Theorem}[section] 
\newtheorem{prop}[thm]{Proposition} 
\newtheorem{lem}[thm]{Lemma} 
\newtheorem{cor}[thm]{Corollary}
\theoremstyle{remark} 
\newtheorem{rem}{Remark}[section] 
\theoremstyle{definition}
\newtheorem*{defin}{Definition}
\numberwithin{equation}{section}

\newcommand*{\id}{\operatorname{Id}}

\newcommand*{\loc}{\mathrm{loc}}
\newcommand{\ii }{{\rm i} }
\usepackage[textsize=tiny,textwidth=27mm]{todonotes}% Allow use of command \todo{} 
\usepackage{graphicx,pifont}

\newcommand*{\bydef   }{\overset{\rm def}{=}}
\newcommand*{\norm}[1]{\left\Vert #1\right\Vert}

\begin{document}
\title[Multi-Layer Quasi-Geostrophic Equations]{Dynamic Behavior of a Multi-Layer Quasi-Geostrophic Model: Weak and Time-Periodic Solutions}

% ============
% = author 1 =
% ============
\author[2]{Zineb Hassainia}
%\address{New York University Abu Dhabi \\
%Abu Dhabi \\
%United Arab Emirates} 
% ============
% = author 2 =
% ============
\author[2]{Haroune Houamed}
\address{New York University Abu Dhabi \\
Abu Dhabi \\
United Arab Emirates}
\email{zh14@nyu.edu, haroune.houamed@nyu.edu}

% \subjclass{}
\date{\today}

\begin{abstract}
The quasi-geostrophic two-layer (QS2L) system models the dynamic evolution of two interconnected potential vorticities, each is governed by an active scalar equation. These vorticities are linked through a distinctive combination of their respective stream functions, which can be loosely characterized as a parameterized blend of both Euler and shallow-water stream functions. 

In this article, we study (QS2L) in two directions: First, we prove the existence and uniqueness of global weak solutions in the class of Yudovich, that is when the initial vorticities are only bounded and Lebesgue-integrable. The uniqueness is obtained as a consequence of a stability analysis of the flow-maps associated with the two vorticities. This approach replaces the relative energy method and allows us to surmount the absence of a velocity formulation for (QS2L).  
Second, we show how to construct $m$-fold time-periodic solutions bifurcating from two arbitrary distinct initial discs rotating with the same angular velocity. This is achieved  provided that the number of symmetry $m$ is large enough, or for any symmetry $m\in \mathbb{N}^*$ as long as one of the initial radii of the discs does not belong to some set that contains, at most, a finite number of elements.   
Due to its multi-layer structure, it is essential to emphasize that the bifurcation diagram exhibits a two-dimensional pattern. Upon analysis, it reveals some similarities with the scheme accomplished for the doubly connected V-states of the Euler and shallow-water equations. However, the coupling between the equations gives rise to several difficulties in various stages of the proof when applying Crandall-Rabinowitz's Theorem. To address this challenge, we conduct a careful analysis of the coupling between the kernels associated with the Euler and shallow-water equations. 
\end{abstract}

\maketitle

 \tableofcontents

% ================
% = Introduction =
% ================
\section{Introduction and main results} 
In this work, we consider the quasi-geostrophic two-layer model (Phillips' model) given by 
\begin{equation}\label{EQ} \tag{QS2L}
\left\{
\begin{array}{ll}
\partial_t \omega_i + (\nabla ^\perp \psi_i)\cdot \nabla \omega_i=0, &\qquad \qquad(t,x)\in \mathbb{R}^+\times \mathbb{R}^2,
\vspace{2mm}\\
\omega_i = \Delta \psi_i +(-\delta)^{i-1} \lambda^2 (\psi_2 - \psi_1),& \qquad\qquad i\in \{1,2\},
\vspace{2mm}\\
\omega_{i}|_{t=0}= \omega_{i,0}, 
\end{array}
\right.
\end{equation} 
 where $\psi_i$  and  $\omega_i$  stand for the stream function and potential vorticity in the $i^{\text{the}}$ layer, respectively. The parameter  $\delta>0$ above refers to the ratio of the upper to lower layer thickness when the fluid is at rest, whereas   $\lambda \geq 0$ describes the rigidity of the interface. In particular, when $\lambda =0$ (which
corresponds to a perfectly rigid interface), the two layers become uncoupled and
behave as two independent two-dimensional systems obeying the Euler equations.

 The system of equations \eqref{EQ} serves as a simplified model and  a foundational concept  in the study of large-scale atmospheric and oceanic flows. In atmospheric applications, these layers often represent the upper and lower troposphere. On the other hand, in oceanic modeling, they might correspond to the upper mixed layer and the deeper ocean. See for instance \cite{PZF89} for more details.

Though it is still far less studied compared to Euler equations, from analytical point of view,  the quasi-geostrophic two-layer model has enjoyed considerable interest in computational fluid dynamics. We refer to \cite{Carton-2010,Carton-2010-2,Carton-2014,flierl_1988,polvani_1991,PZF89} and the references therein for a series of results on the numerical analysis of \eqref{EQ}.  

In the first part of this paper, we are interested in the construction  of weak solutions to the quasi-geostrophic two-layer  system \eqref{EQ} with initial data of Yudovich-type.   These solutions solve a weak formulation of the equations in the distribution sense, which we introduce next. 

\begin{defin}[Weak--distributional solutions]
 We say that  $(\omega_i)_{i\in \{1,2\}}$ is a weak distributional (or simply weak)  solution of \eqref{EQ} with initial data $\omega_{i,0}\in L^p(\mathbb{R}^2)$, for some $p\in [1,\infty]$,  if 
\begin{itemize}
\item [(i)]  $\omega_1, \omega_2 \in L^\infty_{\text{loc}}(\mathbb{R}^+; L^p(\mathbb{R}^2)),$  \vspace{2mm}
  \item [(ii)]  $\nabla^\perp \psi_1, \nabla^\perp \psi_2 \in L^\infty_{\text{loc}}(\mathbb{R}^+; L^{p'}(\mathbb{R}^2)),$ where $p'$ denotes the conjugate of $p$, \vspace{2mm}
  \item [(iii)]  for  any test function $\varphi \in C^1_c(\mathbb{R}^+ \times \mathbb{R}^2)$, it holds that\vspace{2mm}
  $$ \int_{[0,t]\times\mathbb{R}^2} \omega_i(t,x) \left(\partial_t \varphi + \nabla^\perp \psi_i\cdot \nabla \varphi\right) (t,x) dxdt=\int_{  \mathbb{R}^2}\omega _i(t,x)\varphi (t,x)dx - \int_{  \mathbb{R}^2}\omega _i(0,x)\varphi (0,x)dx, $$
\end{itemize}
for all $i\in \{1,2\}$ and   $t\in \mathbb{R}^+$.
\end{defin}  
Questions of global existence and uniqueness of weak solutions have been originally addressed in the case of Euler equations by Yudovich in \cite{Yudovich1, Yudovich2, vishik1999incompressible},  where this received satisfactory answers  provided that  the initial vorticity is bounded or lies in  $L^p$ spaces, for all $p<\infty$, with adequate assumptions on the growth of its norm as $p\to\infty$. Later on,   non-uniqueness was obtained for the forced Euler equations \cite{DEMC21}, where the vorticity is barely bounded, and in the unforced case \cite{BC21} in a weaker functional setting.

  As it is well known, the solenoidal velocity field associated with a bounded vorticity is not Lipshitz in general. Instead, it enjoys the so-called LogLipshitz regularity which is crucial in the original proof of uniqueness for the Euler equations \cite{Yudovich1, Yudovich2} where the approach therein is laid out by a stability analysis in the relative energy setting. This aproach does not seem to be applicable in the case of quasi-geostrophic two-layer model \eqref{EQ} because the equations of the velocities $\nabla ^\perp \psi_1$ and $\nabla ^\perp \psi_2$ are far from been compared to the velocity equation in Euler's system.  

The alternative way to study stability aspects of models such as quasi-geostrophic two-layer \eqref{EQ} would be by  understanding first the same question about a different quantity. Here, this will be done for  the flow-maps associated with the velocities $\nabla ^\perp \psi_1$ and $\nabla ^\perp \psi_2$. This draws insight from the recent work \cite{CRIPPA1} on nonlinear transport equations advected by a non-local velocity field. This also motivates the consideration  Lagrangian solutions that we recall next.

\begin{defin}[Lagrangian solutions]
We say that $(\omega_i)_{i\in \{1,2\}}$ is a   Lagrangian solution of \eqref{EQ} if   $\omega_i(t,\cdot)$ is the push forward of the initial data $\omega_{i,0}$ by the flow-map associated with the velocity field $\nabla ^\perp \psi _i$, for any $i\in \{1,2\}$. More precisely,   $(\omega_i)_{i\in \{1,2\}}$ is a   Lagrangian solution of \eqref{EQ}  if 
$$\omega_i(t,\cdot) = X_i(t,\cdot)_{\sharp} \omega_{i,0}$$
   and $X_i(t,\cdot) $ solves the ODE
\begin{equation}\label{ODE}\tag{ODE}
  \left\{ \begin{array}{ll}
  \frac{d }{dt}X_i(t,x) = (\nabla ^\perp \psi _i)\left( X_i(t,x) \right),\vspace{2mm}\\
  X_i(0,x) = x,  
\end{array}
\right.
\end{equation}
for all $i\in \{1,2\}$ and $(t,x)\in \mathbb{R}^+\times\mathbb{R}^2.$
\end{defin}

In order to construct global solutions of \eqref{EQ} in Yudovich's class, it is important to understand the coupling in that system of equations. A crucial step in our analysis below is introducing   suitable new unknowns which are made of a specific combination of the solutions $\omega_1$ and $\omega_2$. The new unknowns allow us to recast the equations connecting the vorticities with their stream functions in a way that is naturally compatible with the coupling in \eqref{EQ} and which also reveals the connexion between \eqref{EQ}, Euler and shallow-water equations. This is fundamental in our proof of existence of solutions and is discussed in Section \ref{section:Yudovich}. Afterwards, this observation inspires introducing a similar combination of the flow--maps $X_1$ and $X_2$ which serves in their stability analysis. The uniqueness of weak-distributional solutions in Yudovich's class follows subsequently. This is summarized in our first main theorem which we state next.

 \begin{thm}[Existence and uniqueness of weak   solutions]\label{Thm:1} For any initial data satisfying  
 $$(\omega_{1,0}, \omega_{2,0})\in L^1\cap L^\infty(\mathbb{R}^2), $$  there is a unique global weak solution  $(\omega_1, \omega_2)$ to  \eqref{EQ}  enjoying the bounds  
$$ (\omega_1, \omega_2) \in L^\infty(\mathbb{R}^+, L^1 \cap L^\infty(\mathbb{R}^2)) ,  $$  
as well as the conservation of norms
$$ \norm {\omega_j(t,\cdot)}_{L^q(\mathbb{R}^2)} = \norm {\omega_{j,0}}_{L^q(\mathbb{R}^2)}, $$
 for any   $i\in \{1,2\}$,  $q\in [1,\infty]$ and $t\geq 0$.
   Moreover, the solution  is continuous in time in the sense that 
  \begin{equation*}
  	(\omega_1, \omega_2) \in C(\mathbb{R}^+, L^p(\mathbb{R}^2))   ,  
  \end{equation*}
for all $p\in [1,\infty)$. 
 \end{thm} 
 
Before we move on to our second main result of this  paper, allow us to establish a remarkable consequence of Theorem \ref{Thm:1}. In particular, the next corollary exhibits the connexion between solutions of the system of equation \eqref{EQ} in the case $\delta=1$ and the Euler equations.
 
   \begin{cor}\label{coro:1}
 The solution $ (\omega_1,\omega_2)$ of the system of equations \eqref{EQ} with  $\delta=1$ and initial data $$\omega_{0,1}= \omega_{0,1}\bydef \omega_0 \in L^1\cap L^\infty(\mathbb{R}^2)$$
is given by 
$$ \omega_1=\omega_2=\omega,$$
where 
$\omega$ is the unique Yudovich solution of the Euler's equation
\begin{equation}\label{Euler}
\left\{
\begin{array}{ll}\tag{E}
\partial_t \omega  + v  \cdot \nabla \omega =0,\\
v  =-  \nabla ^\perp  (-\Delta)^{-1} \omega ,\\
\omega_{|t=0} = \omega_0.
\end{array}
\right.
\end{equation}
 \end{cor}
 \begin{proof}
 The proof is straightforward once we notice that the system of equations \eqref{EQ} is symmetric when $\delta=1$ in the sense that, if $(\omega_1,\omega_2) $ is  the solution  associated with  the initial data $(\omega_{0,1},\omega_{0,2}) $, then $(\omega_2,\omega_1) $ is the   solution associated with the initial data $(\omega_{0,2},\omega_{0,1}) $. Accordingly, if we set initially     
 $$\omega_{0,1}= \omega_{0,1},$$
 then  it   immediately   follows, by uniqueness of solutions from Theorem \ref{Thm:1},  that 
 $$(\omega_2,\omega_1) = (\omega_{ 1},\omega_{ 2})  . $$ 
 Consequently, we obtain from the second equation in \eqref{EQ} that 
 $$  \left(- \Delta + \lambda^2 \right) (\psi_1 -\psi_2)=0, $$
 whereby, we deduce that
$$ \psi _1 = \psi_2.$$
Therefore, by inserting the latter identity in \eqref{EQ}, we conclude that $\omega\bydef    \omega_1=\omega_2$ solves the Euler equation \eqref{Euler}, thereby completing the proof of the corollary.
 \end{proof}

Typical examples of solutions covered by Theorem \ref{Thm:1} are the so-called vortex-patches. These are vorticities uniformly distributed in a bounded region $D$ and are generated from given initial data of the format 
 $$ (\omega_{1,0}, \omega_{2,0})= (\mathds{1}_{D_1}, \mathds{1}_{D_2}), $$
  where $D_i$, for $i\in \{1,2\}$, is a bounded domain in $\mathbb{R}^2$ with smooth boundary.
  These initial data fall in the setting of Theorem \ref{Thm:1} which ensures the existence of a unique global solution associated with them. Moreover, the patch structure is preserved by the evolution and, at each time $t\geq 0$, the potential vorticities are  given by 
  \begin{equation*}
  	(\omega_{1}, \omega_{2})= (\mathds{1}_{D_{1,t}}, \mathds{1}_{D_{2,t}}),
  \end{equation*} 
 with, for all $i\in \{1,2\}$, the transported domain 
 \begin{equation*}
 	D_{i,t}\bydef X_i(t, D_{i,t})
 \end{equation*}
  denotes the image of $D_{i}$ by the flow $X_i$ defined as the solution of \eqref{ODE}.  
Note that, when $D_1$ and $D_2$  are discs, then, the potential vorticities  are stationary solutions to \eqref{EQ}.  Thus, it is natural  to look for periodic solutions close to these  

The second aim of this paper is to study the existence of time-periodic solutions of \eqref{EQ} given by rigid body rotating vortex patches around the origin which are described by
$$
(D_{1,t} , D_{2,t}) = ( e^{\ii \Omega t} D_{1} , e^{\ii \Omega t}D_{2}) ,
$$ 
for some time-independent parameter $\Omega$ referring to as the angular velocity.  
The potential vorticities corresponding to the preceding domains are then stationary  along the rotating frame with angular velocity $\Omega$ and the boundaries 
$\partial D_{1,t}$ and $ \partial D_{1,t} $  evolve according to the non-linear, non-local, system 
 \begin{equation}\label{vortex patch equation}
  \left(\nabla ^{\perp} \psi_i (t,x)- \Omega x^\perp \right) \cdot \vec{n}_i(x) = 0,
  \end{equation}
 for all $ x\in  \partial D_{i}$ and $i\in \{1,2\}$, where $\vec{n}_i$ is the outward normal unit vector associated with the initial boundary   $\partial D_i$.  Such solutions are called V-states and have been explored, first, numerically  in the case of Euler's equations  by Deem and Zabusky \cite{DZ78}. The analytical proof of their existence is due to Burbea \cite{B82} and relies on the contour dynamics equation, the conformal mapping parametrization and  the celebrated local bifurcation theorem of Crandall and Rabinowitz \cite{CR71}. The outcome is the existence of a countable family of local curves of V-states with $m$-fold symmetries (i.e. invariant by $\frac{2\pi}{m}$ angular rotation), with $m\geq 2$, bifurcating from the disc at the angular velocities  $\Omega_m=\tfrac{m-1}{2m}$. A global continuation of these local curves were constructed in  \cite{HMW20}, where the global curves limit to a vanishing of the angular fluid velocity. 
 
Other uniformly rotating vortex patch solutions to Euler's equations were constructed,  using Burbea's techniques, close to the annulus in \cite{HdelaHMV}  and close to the Kirchoff ellipses in \cite{MR3462104,hmidi2015bifurcation}.  See also \cite{Guo04} for an instability result of  Kirchoff's ellipses.  

  In the last few years, there have been several investigations on the V-states in different settings, such as the study of the boundary regularity of the V-states \cite{HMV13}, the existence of  multipole vortex patches \cite{G20,G21,HH21,HM17,HW22} and the radial symmetry properties of stationary and uniformly-rotating solutions \cite{GSPSY-rigidity,H15,GSPS:2021}. Very recently  quasi-periodic patch solutions to Euler's equations  have been constructed  using the Nash-Moser scheme and KAM theory, see \cite{MR4623544,HHR23, HR22}.
 
Similar results to the ones mentioned above were also obtained  for other active
scalar equations such as the generalized surface quasi-geostrophic equation, the quasi-geostrophic shallow-water
equations and Euler equations on the rotating unit 2-sphere. We refer to  \cite{GSIP23, HHM21,R22,HR21,GHR23} and the references therein for a series of relevant results.

In the second part of this paper,  we adapt the techniques from \cite{B82, hmidi2015bifurcation} and show how to extend their validity to build time-periodic solutions of \eqref{EQ}. In particular, our second result consists in establishing the existence of $m$-fold symmetric V-states for \eqref{EQ}.

We point out in passing that there have been some numerical simulations suggesting the existence of V-state solutions for \eqref{EQ} in some specific cases, see \cite{polvani_1991, PZF89} for instance. In the present study, we provide an analytical proof of their existence through the application of local bifurcation theory. Informally stated,  our second main result is outlined in the following  theorem. A slightly more detailed and precise version will be discussed  in Section \ref{section:V-states}, later on.

 \begin{thm}[Periodic solutions bifurcating from simple eigenvalues]\label{Thm:2-A}
  Let $m\in\mathbb{N}^*$, $\delta  \geq 1$,  $\lambda>0$ and $0<b_2\leq  b_1$ be a set of real numbers. 
 Then, there exist two curves of $m$-fold symmetric pairs of   simply connected V-states solving \eqref{EQ} and
bifurcating from the stationary states
  \begin{equation}\label{2discs}
  \omega_1= \mathds{1}_{\{x\in \mathbb{R}^2 :  |x| < b_1\}} \quad \text{and} \quad \omega_2= \mathds{1}_{\{x\in \mathbb{R}^2 :  |x| < b_2\}}
  \end{equation}   if  one of the following two conditions is fulfilled:
\begin{itemize}
	\item either $b_1\neq b_2$ and  the number of symmetry $m$ is large enough, i.e., $m\gg 1$,
	\item or $m\in \mathbb{N}^*$, $b_1\in(0,\infty)$ and  $ b_2\not\in S_{m,b_1}$ for some  set  $S_{m,b_1}\subset(0,b_1]$ containing, at most, a finite number of elements. 
	 \end{itemize}
 \end{thm}

 \begin{rem} 
It is to be emphasize that periodic solutions with $m$-fold symmetry, for any $m\in \mathbb{N}^*$, bifurcating from the two discs  \eqref{2discs},  exist for almost all values of $ b_1,b_2\in (0,\infty)$. However, there are some values of these radii (for instance some values of $b_1=b_2$, as we will justify later on) where the celebrated Crandall-Rabinowitz’s Theorem does not directly apply due to existence of ``spectral collisions''. In such cases, it is not clear yet, at least by the analysis in this work, if time-periodic solutions do exist.
 \end{rem}
 
 \begin{rem}
 Although solutions of \eqref{EQ} are not  stable in general by switching the positions of the vorticities---unless if we consider the case $\delta=1$ as is shown in Corollary \ref{coro:1}---our elements of proof of Theorem \ref{Thm:2-A} remain unchanged  in the case $0< b_1 \leq b_2$ and the same bifurcation result holds true, as well.  In addition, the restriction on $\delta\geq 1$ can in fact be relaxed, see Remark \ref{RMK:bi}, below. 
   \end{rem}
    \subsubsection*{Structure of the paper}
 The remaining sections of this paper are organized as follows:
 The proof of existence and uniqueness of global weak solutions (Theorem \ref{Thm:1}) is the subject of Section  \ref{section:Yudovich}. This is where the crucial coupling in the set of equations \eqref{EQ} will be discussed in details, altogether with the approximate scheme of \eqref{EQ} that we utilize to build global weak and Lagrangian solutions.  
Afterwards,  a stability result of the flow-maps associated with the weak solutions will be established which eventually yields the uniqueness of weak solutions. 

 At last, in Section \ref{section:V-states}, we build  time-periodic solutions by employing     Crandall-Rabinowitz’s Theorem \ref{Crandall-Rabinowitz theorem}. This requires a precise analysis of spectral properties of the linearized operator associated with the boundary equation around steady states (discs) which is discussed with details in the same that section.
 
 The achievement of these results also hinges upon a careful understanding of several properties of   Laplace and shallow-water kernels, as well as a fine analysis of Bessel functions. In Sections \ref{section:black-box}, below, we provide the reader with a tool box of various abstract results on functional analysis and operator theory that will be employed in the proof of our theorems, and which may also serve in other different contexts.

 \section{Functional tools and building blocks}\label{section:black-box}
 Before we step over the elements of proof of our main theorems, we devote this section to a short self-contained introduction of the basis analysis of functional and operational settings as well as properties of the relevant kernels and all notions that will play a role in the construction of weak and time-periodic solutions in the upcoming sections. 
 
 In the sequel, we are going to use classical notations for functional spaces, such as Lebesgue spaces, space for log-Lipschitz functions \dots.  Moreover, for the sake of simplicity, we will often utilize the symbol $\lesssim$ (resp. $\lesssim_\delta$) instead of $\leq C$ (resp. $\leq C_\delta $) when the dependance on the constant $C$ (resp. the constant $C_\delta$) is harmless (resp. degenerates when the parameter  $\delta$ approaches its endpoint values).

\subsection{Bessel functions and asymptotic expansions}\label{section:Bessel}
Here, we recall a few important identities that we will later use to describe the decay properties of specific parameterized integrals. These identities, among more relevant other results, can be found in the books by Abramowitz and Stegun \cite{Abram} and by Watson \cite{W95}.

We first recall that  the Bessel function of the first kind and order $\nu \in \mathbb{N}$  given by the expansion	    
\begin{equation}\label{def modified Bessel function of first kind}
J_{\nu}(z)\bydef \sum_{m=0}^{\infty}\frac{(-1)^m\left(\frac{z}{2}\right)^{\nu+2m}}{m!\Gamma(\nu+m+1)}, \quad \text{for all} \quad |\mbox{arg}(z)|<\pi.
\end{equation}	 
In particular,  when $\nu=n\in\mathbb{N}$, the preceding Bessel function admits the following integral representation, which can be found in \cite[Identity 9.1.21]{Abram},  
\begin{equation}\label{integral representation Jn}
 J_n(z)=\frac{1}{\pi}\int_0^\pi\cos\big(n\theta-z\sin(\theta)\big)d\theta, \quad \text{for all} \quad z\in\mathbb{C}.
\end{equation}    
On the other hand, the Bessel functions of imaginary argument are given by
\begin{equation}\label{In:def}
I_{\nu}(z)\bydef \sum_{m=0}^{\infty}\frac{\left(\frac{z}{2}\right)^{\nu+2m}}{m!\Gamma(\nu+m+1)}, 
\end{equation}
 for all $z\in \mathbb{C}$ with $|\mbox{arg}(z)|<\pi$.
 Thus, for all  $\nu\in\mathbb{C}\backslash\mathbb{Z}$, we define the $K$-Bessel function by
  $$K_{\nu}(z)\bydef \frac{\pi}{2}\frac{I_{-\nu}(z)-I_{\nu}(z)}{\sin(\nu\pi)} $$
and, when $n\in\mathbb{Z},$ we set  
$$K_{n}(z)\bydef \displaystyle\lim_{\nu\rightarrow n}K_{\nu}(z) $$ 
for all $z\in \mathbb{C}$ such that $|\mbox{arg}(z)|<\pi$.  Moreover, we emphasize that 
\begin{equation*}
	I_{-n} \equiv I_n  \qquad \text{and} \qquad K_{-n}\equiv K_n,
\end{equation*}
for all $n\in \mathbb{N}$ and that $ I_n$ and $K_n$ are smooth functions away from the origin and satisfy 
\begin{equation}\label{K:diff}
	I'_0(z) =  I_{1}(z) \qquad \text{and} \qquad K'_0(z) = - K_{1}(z),
\end{equation} 
 for all $z \in  \mathbb{C} $ such that $|\mbox{arg}(z)|<\pi$. The preceding identities can be found in   \cite[pages 357--356]{Abram}.  

The Bessel functions $I_n$ and $K_n$ enjoy several properties and are featured by different equivalent representation formulas. Here, we  recall a useful expansion for $K_n$ that will be utilized in our proof below, which can be found  in \cite[Identity 9.6.11]{Abram}  
\begin{equation*}
	\begin{aligned}
		K_{n}(z) 
		&=\frac{1}{2}\left(\frac{z}{2}\right)^{-n}\sum_{k=0}^{n-1}\frac{(n-k-1)!}{k!}\left(\frac{-z^2}{4}\right)^{k}+(-1)^{n+1}\log\left(\frac{z}{2}\right)I_{n}(z)
		\\
	    & \quad +\frac{1}{2}\left(\frac{-z}{2}\right)^{n}\sum_{k=0}^{\infty} \frac{\Phi(k+1)+\Phi(n+k+1)}{k!(n+k)!} \left(\frac{z }{2}\right)^{2k},   
	\end{aligned}
\end{equation*} 
for any $n\in \mathbb{N}$ with $z \in  \mathbb{C} $ such that $|\mbox{arg}(z)|<\pi$, where  $\psi(1)=-\boldsymbol{\gamma} $ stands for Euler's constant, and we set, for any $m \in \mathbb{N}^*,$ that 
$$ \Phi(m+1)=\displaystyle\sum_{k=1}^{m}\frac{1}{k}-\boldsymbol{\gamma}.$$
In particular, one has the useful identity
\begin{equation}\label{expan K0}
K_{0}(z)=-\log\left(\frac{z}{2}\right)I_{0}(z)+\sum_{m=0}^{\infty}\frac{\left(\frac{z}{2}\right)^{2m}}{(m!)^{2}}\Phi(m+1),
\end{equation}
for any $z \in  \mathbb{C} $ with $|\mbox{arg}(z)|<\pi$.

 The following lemma summarizes several key properties of the Bessel functions. Note that most of the results in the next lemma  are established in previous works, whence, we only outline the idea to justify the ones that we were not able to find  in the literature.  
    \begin{lem}\label{lemma InKn} Let $x,y\in (0,\infty)$ be fixed such that $x\leq y$. Then, the following holds:
    
    \begin{enumerate}
    	\item The sequence $\big((I_n K_n)(x)\big)_{n\geq 1}$  is strictly decreasing and converges to $0$. More generally, the mapping 
    $$(n,x) \mapsto (I_nK_n)(x) $$
     is strictly decreasing on $ \mathbb{N} \times \mathbb{R}^+$.
     \item The sequence $\left(\tfrac{(\frac{x}{y}) ^n}{2n}-I_n(x)K_n(y)\right)_{n\geq 1}$  is positive, decreasing and converges to zero. Moreover, it holds that 
     \begin{equation}\label{asymptotic:Bessel:1}
    	0<\frac{1}{2n}\left(\frac{x}{y}\right) ^n-I_n(x)K_n(y) \leq   \frac{1}{2n} ,
    \end{equation}
    for all  $n\in \mathbb{N}^*$.
    \item The function 
   $$ 
   x\mapsto \frac{ I_1(x)}{x} 
    $$ 
    is strictly increasing on $(0,\infty)$. 
    \end{enumerate} 
    \end{lem}
\begin{proof}
The proof of the first claim of the lemma is the subject of \cite{Baricz} and  \cite{Segura}. Now,
 in order to justify the second claim, we begin with writing that
\begin{equation*}
	I_n(x)K_n(y)=\frac{1}{2}\int_{\log \frac{y}{x}}^\infty J_0\left(\mu\sqrt{2 xy\cosh(t)-x^2-y^2)}\right) e^{-nt}dt,
\end{equation*}
where $J_0$ is given by \eqref{def modified Bessel function of first kind}.
The latter identity can be found in \cite[page 140]{Lebedev}. 
Therefore, one deduces, for any $n\in \mathbb{N}^*$, that   
\begin{equation}\label{integ rep InKn}
	\frac{(\frac{x}{y}) ^n}{2n}-I_n(x)K_n(y)=\frac{1}{2}\int_{\log \frac{y}{x}}^\infty\Big(1-J_0\left( \sqrt{2xy\cosh(t)-x^2-y^2)}\right)\Big)e^{-nt}dt.
\end{equation}
 On the other hand, using the integral representation \eqref{integral representation Jn}, we observe that
\begin{align*}
1-J_0\left( \sqrt{2xy\cosh(t)-x^2-y^2)}\right)&=\frac{1}{\pi}\int_0^\pi\Big[1-\cos\left( \sqrt{2xy\cosh(t)-x^2-y^2)}\sin\theta\right)\Big]d\theta\geq 0.
\end{align*}
Therefore, this implies that the sequence $\left(\tfrac{(\frac{x}{y}) ^n}{2n}-I_n(x)K_n(y)\right)_{n\geq 1}$  is positive and  decreasing.

Next,  the asymptotic decay  \eqref{asymptotic:Bessel:1},  directly follows from the fact that 
\begin{equation*}
	I_n(x)K_n(y)>0, 
\end{equation*} 
 for all  $n\in \mathbb{N}$,  and the assumption that $x\leq y$. 
 We now turn  to justify that the function $x\mapsto \frac{ I_1(x)}{x}$ is strictly increasing on $(0,\infty)$. To that end, we simply need to notice, in view of the expansion formula \eqref{In:def}, that 
 \begin{equation*}
 	\frac{I_1(x)}{x} = \frac{1}{2}\sum_{ m=0}^\infty  \frac{\left(\frac{x}{2}\right)^{ 2m}}{m! (m+2)!},\quad \text{for all} \quad  x \in (0,\infty). 
 \end{equation*}
 The function on the right-hand side above is clearly increasing on $(0,\infty)$, which concludes the justification of the last claim of the lemma and completes its proof.
\end{proof}

 \subsection{Kernel estimates}

In this paragraph, we intend to show the continuity and   boundedness  of convolution-type operators that   naturally   appear through a specific combination between   velocities  associated with  \eqref{EQ}. But before that, let's introduce the kernels of our interest by first recalling that distributional solutions of 
\begin{equation*}
	-\Delta \textbf{G} = \delta_0 , \quad \text{in} \quad \mathcal{S}'(\mathbb{R}^2)
\end{equation*}
and, for any $\varepsilon>0$, 
\begin{equation*}
	(-\Delta + \varepsilon^2) \textbf{G} _\varepsilon = \delta_0 , \quad \text{in} \quad \mathcal{S}'(\mathbb{R}^2)
\end{equation*}
are respectively given by 
\begin{equation*}
	{\bf{G}}(x)\bydef  -\frac{1}{2\pi} \log|x| \qquad \text{and} \qquad {\bf{G}}_{\varepsilon}(x)\bydef   \frac{1}{2\pi} K_0(\varepsilon |x|),
\end{equation*}
where $K_0$ is the modified Bessel function of zero order introduced in the previous section. Indeed, the first fundamental solution is a classical fact, see for instance the book by Evans \cite[Section 2.2]{E98}. 
As for the justification for the fundamental solution of the second problem above, we refer to \cite[Section 4.2]{DTC19}. We now consider the system of equations 
\begin{equation*}
	\begin{aligned}
		\omega_1 & = \Delta \psi_1 + \lambda^2 (\psi_2 - \psi_1),
		\\
\omega_2 &= \Delta \psi_2 + \delta\lambda^2 (\psi_1- \psi_2),
	\end{aligned}
\end{equation*}
for given real parameters $\delta, \lambda \in (0,\infty)$ and two functions $\omega_1$ and $\omega_2$ with appropriate decay at infinity\footnote{In the case of the present paper, $\omega_1$ and $\omega_2$ are assumed to belong to $L^1\cap L^\infty(\mathbb{R}^2)$, which is enough to give a sense to all the computations in this section.}. Therefore, it is readily seen that 
\begin{equation*}
	\delta \omega_1 + \omega_2 = \Delta ( \delta \psi_1 + \psi_2)
\end{equation*}
and that 
\begin{equation*}
	  \omega_1 - \omega_2 = \Big( -\Delta    + (\delta+1)\lambda^2 \Big)(\psi_2-\psi_1).
\end{equation*}
Hence, by employing the fundamental representation of Laplace and shalow--water operators, we arrive at the identities
\begin{equation*}
	\delta \psi_1 + \psi_2 =  -\int_{\mathbb{R}^2}{\bf{G}}(\cdot-\xi)\big( \delta {\omega}_1 + \omega_2 \big) (\xi)d\xi
\end{equation*}
and
\begin{equation*}
	 \psi_2-\psi_1 =  \int_{\mathbb{R}^2}{\bf{G}}_{\mu}(\cdot-\xi)\big(   {\omega}_1 -  {\omega}_2 \big)(\xi)d\xi,
\end{equation*}
where we denote 
\begin{equation*}
	\mu \bydef  \lambda\sqrt{1+\delta}  .
\end{equation*}
It is then readily seen that the previous two identities lead to the following representation 
\begin{align}\label{def streamL1}
	\psi_k(z)&=\sum_{j=1}^2\int_{\mathbb{R}^2}G_{k,j}(z-\xi){\omega}_j(t,\xi)dA(\xi),
	\end{align}
	 where we set 
	\begin{align}\label{def Gkj}
	G_{k,j}(x)&\bydef   \frac{\delta^{2-j}}{2\pi(\delta+1)}\log |x|+(-1)^{k+j-1} \frac{\delta^{k-1}}{2\pi(\delta+1)} K_0(\mu |x|), 
		\end{align} 
	for any $x\in \mathbb{R}^2 \setminus \{0\}$ and $k\in \{1,2\}$. This representation will come in handy in the construction of time-periodic solutions, later on. 
	
	The short introduction above motivates   the study of the kernels ${\bf{G}}$ and ${\bf{G}_{\varepsilon}}$. For a later use in the proof of existence of weak solutions, we will also need to prescribe some fine properties of the kernels associated with  the operators  
	$$
	\nabla ^\perp (-\Delta) ^{-1} \qquad \text{and} \qquad  \nabla ^\perp \big(\Delta - \mu^2\big)^{-1},$$
which we respectively denote, from now on, by $k_+$ and $k_- $ (with $\mu=1$ for simplicity). More precisely, we now set, for all $x\in \mathbb{R}^2\setminus \{0\}$, that 
\begin{equation}\label{k:pm:def}
	k_+(x )\bydef  - \frac{1}{2\pi }  \frac{x^\perp}{|x|^2} \qquad \text{and} \qquad 	k_-(x) = -  \frac{x^\perp}{|x|} K_1(|x|) ,
\end{equation} 
   where,    $K_1= K_0'$ denotes the Bessel function of order one, previously introduced in Section \ref{section:Bessel}. Accordingly, we define the convolution operators
$$ K_{\pm}f \bydef    k_{\pm}\star f,$$ 
for any suitable function $f$ for which the preceding convolutions have a sense.

The next lemma establishes quantitative estimates on the behavior of the kernels introduced above. This will serve later on to deduce essential properties of operator defined through a convolution with the preceding kernels which will be useful in the proof of uniqueness of solutions to \eqref{EQ} in Yudovich's class.

\begin{lem}\label{lemma:kernels}
 Let $ k_\pm$ be given by \eqref{k:pm:def}. Then, it holds  that  
\begin{equation}\label{A1}
|k_\pm(x)|\leq \frac{C_1}{|x|}, 
\end{equation}
and 
\begin{equation}\label{A2}
|k_\pm(x) - k_\pm(y)|\leq  C_2\frac{|x-y|}{|x| |y|}.
\end{equation}
for   all $ x,y\in \mathbb{R}^2 \setminus  \{0\}$, for  some universal constants $C_1,C_2>0$.
\end{lem}

\begin{proof}
 Because the bounds on  $k_+$ above are classical, we will only focus on the proof of these bounds  on $k_-$ which we now split into two cases:

\subsubsection*{Case $|x|\leq 1$}
Here, we employ the observation that  $k_- (x)$ can be seen as a perturbation of $ |x|^{-1},$ for finite values of $x\neq 0$. More precisely, we write, in view of \eqref{K:diff} and \eqref{expan K0}, for any $r>0$, that   
\begin{equation*} 
   -K_1 (r) = K_0'(r) =  - \frac{1}{r} +  \frac{r}{2}  \underbrace{ \sum_{m=1}^\infty  m\frac{\left(\frac{r}{2}\right)^{2(m-1)}}{(m  !)^2} \left( \sum_{i=1}^{m-1} \frac{1}{i} - \gamma -    \log \left(\frac{r}{2}\right)   \right) }_{\bydef    H(r)}  ,
\end{equation*}
where, $\gamma$ denotes Euler's constant.  
Therefore, noticing, for any $r\in (0,1]$, that   
$$|r  K_0'(r) | \lesssim 1 , $$
we deduce that \eqref{A1} holds for all $x\in \mathbb{R}^2\setminus \{0\}$ with $|x|\leq 1$.

\subsubsection*{Case $|x|>1$} Now, we     utilize the integral representation of $K_1$ (see \cite[Identity 9.6.23]{AS64}) to write, for any $r>0$, that
\begin{equation}\label{K:integral}
 K_1(r) = \frac{\pi^\frac{1}{2}}{2 \Gamma\left(\frac{3}{2}\right)} r \int_1^\infty e^{-rt} (t^2-1)^\frac{1}{2} dt.
\end{equation}
In view of that, we claim that $K_1$ has an exponential decay as $r\to \infty$. More generally, for a later use, we will now show that the sequence of functions
\begin{equation}\label{Rn:def}
r\mapsto  R_n(r) \bydef     \int_1^\infty  t^n e^{-rt} (t^2-1)^\frac{1}{2} dt ,  
\end{equation} 
has an exponential decay at infinity, for any $n\in \mathbb{N}$. That is, we now claim that 
\begin{equation}\label{decay:Rn}
| R_n(r)| \lesssim_n e^{-r}  ,
\end{equation}   
 for any $r\geq 1$ .
The justification of the preceding bound is achieved by recurrence and simple integration by parts. To see that, we first write, for any $n\in \mathbb{N}$ and all $r\geq 1$, that 
$$ \begin{aligned}
R_n(r)  & = \int_1^2  t^n e^{-rt} (t^2-1)^\frac{1}{2} dt + \int_2^\infty  t^n e^{-rt} (t^2-1)^\frac{1}{2} dt \\
& \leq 2 ^{n+1}   \int_1^ 2  e^{-rt}  dt + \int_2^\infty  t^n e^{-rt} (t^2-1)^\frac{1}{2} dt \\
& \leq 2 ^{n+1}     e^{- r}   + \int_2^\infty  t^n e^{-rt} (t^2-1)^\frac{1}{2} dt  .
\end{aligned}$$
Thus, by   integration by parts  and   using the elementary inequality 
$$ t \leq \sqrt{2} \sqrt{ t^2 -1},$$
which is valid  for all $t \geq 2$,  we find that 
$$ \begin{aligned}
R_n(r)   & \leq  2 ^{n+2}     e^{- r}   + \frac{n+2}{r}   \int_2^\infty  t^{n-1} e^{-rt} (t^2-1)^\frac{1}{2} dt \\
& \leq 2 ^{n+2}     e^{- r}   + (n+2) R_{n-1}(r)  .
\end{aligned}$$
Therefore, it is readily seen   that \eqref{decay:Rn} follows by induction as soon as we show, for all $r\geq 1$, that 
$$  \int_2^\infty  e^{-rt} (t^2-1)^\frac{1}{2} dt \lesssim e^{-r} .$$
To that end, we perform one more integration by parts to obtain, for all $r\geq 1$, that 
$$  \begin{aligned}
 \int_2^\infty   e^{-rt} (t^2-1)^\frac{1}{2} dt &= \frac{\sqrt{3}}{r}  e^{-r} + \frac{1}{r}  \int_2^\infty   e^{-rt} (t^2-1)^{-\frac{1}{2}} dt\\
 & \leq  \frac{\sqrt{3}}{r}  e^{-r} + \frac{1}{\sqrt{3} r}  \int_2^\infty   e^{-rt}  dt,
\end{aligned}$$
which clearly leads to the desired bound.  
 At last, we deduce that 
$$|k_-(x)|\lesssim \frac{1}{|x|} , $$
for any $x\in \mathbb{R}^2$ with $ |x|\geq  1$, thereby concluding the proof of \eqref{A1}.

Now, in order for us to establish \eqref{A2}, we proceed again in two steps. Note that w we can assume that $|y|\leq |x|$ without any loss of generality.

\subsubsection*{Case $ |y| \leq|x| \leq  2|y|$}
 By employing the integral representation \eqref{K:integral}, we obtain that 
$$ \left| k_-(x) - k_-(y) \right| \lesssim  |x-y| R_0(|x|) + |y| \big| R_0(|x|) - R_0(|y|)\big| ,$$
where $(R_n)_{n\in \mathbb{N}}$  is defined in \eqref{Rn:def}. 
 On the one hand, due to the decay estimate \eqref{decay:Rn} and the assumption    $|x|\geq  |y|$, it follows that 
$$ |x-y| R_0(|x|) \lesssim \frac{|x-y|}{|x||y|} ,$$
as soon as $|x|\geq |y|>0$. 
On the other hand,   writing, in view of Taylor's expansion, that
$$|R_0(|x|) - R_0(|y|) | \lesssim   \big|| x|- |y|\big |   \int_0^1 \int_{1}^{\infty} t e^{-t\big(s|x| + (1-s)|y|\big)}  \sqrt{t^2-1} dt ds $$
and  exploiting the assumption $ |x|\geq  |y|$, again, we obtain  that
$$\begin{aligned}
|R_0(|x|) - R_0(|y|) | &\lesssim   \big|| x|- |y|\big |   \int_{1}^{\infty} t e^{-t |y| }  \sqrt{t^2-1} dt \\
  & =   \big|| x|- |y|\big | R_1 (|y|).
\end{aligned} $$
Therefore, by further employing the decay estimate \eqref{decay:Rn} and the assumption $ |x|\leq 2 |y|$, we arrive at the bound 
$$\begin{aligned}
|y||R_0(|x|) - R_0(|y|) | &\lesssim \frac{| x - y|  }{|y|^2} \\
&\lesssim \frac{| x - y|  }{|x||y|}, 
\end{aligned} $$
thereby showing the validity of \eqref{A2} whenever $x,y\in \mathbb{R}^2$ with $ |y|\leq |x| \leq 2|y|$. 
\subsubsection*{Case $ |y| <   2|y| \leq|x| $}
We first proceed in a similar way to   the previous case by writing
$$ \left| k_-(x) - k_-(y) \right| \lesssim  |x-y| R_0(|x|) + |y|  \Big(  R_0(|x|) +  R_0(|y|)\Big) .$$
Hence, by virtue of the decay estimate \eqref{decay:Rn}, we obtain, as long as $ |y|\leq |x|$, that 
$$\begin{aligned}
 \left| k_-(x) - k_-(y) \right|  &\lesssim   \frac{|x-y| }{|x|^2} +    \frac{|y|  }{|x|^2} + \frac{1}{|y| } \\ 
 &\lesssim   \frac{|x-y| }{|x| |y|} +     \frac{1}{|y| }.
\end{aligned}$$
At last, further employing the assumption  $   2|y| \leq|x| $ entails that
$$  \begin{aligned}
 \frac{1}{|y| }  &\leq 2 \left( \frac{1}{|y|} - \frac{1}{|x|}\right) \\
 &\leq   2  \frac{ |x-y|}{|x||y|} ,
\end{aligned} $$
thereby yielding the desired control \eqref{A2} in the case where $   2|y| \leq|x| $, as well. 
 All in all, combining the foregoing estimates leads to a complete justification of \eqref{A2} and concludes the proof of the lemma.
 \end{proof}

 As a consequence of the preceding lemma, we now state a lemma which will be  essential in the proof of the uniqueness   in Theorem \ref{Thm:1}. The proof of the next lemma can be justified by a direct application of  \cite[Corollary 2.4]{CRIPPA1} and \cite[Remark 2.5]{CRIPPA1} that only require suitable quantitative bounds on the kernels $k_{\pm}$, which we already established in Lemma \ref{lemma:kernels}, above.
 
Here and in what follows, the  real function $\ell$ is defined as follows 
$$\ell: [0,\infty) \rightarrow [0,1]$$
where we set
\begin{equation}\label{ell:def}
 \ell(r) \bydef    \left\{ \begin{array}{ll}
 0, &  \text {for  }  r=0,\\
r \log \left( \frac{e}{r}\right),  & \text {if } r\in (0,1]\\
1, & \text{elsewhere}.
\end{array}
\right.
\end{equation} 

\begin{lem}\label{kernel:ES} For any $f\in L^1\cap L^\infty 
(\mathbb{R}^2),$ it holds that 
$$\norm{K_{\pm}f}_{L^\infty}\lesssim \norm f_{L^1\cap L^\infty},$$
where $K_\pm $ are the convolution operator associated with the kernels $k_\pm$ defined in \eqref{k:pm:def}.
 Moreover,  we have that 
$$ \int_{\mathbb{R}^2} \left| k_{\pm} (x-z)-   k_{\pm} (y-z)\right| |f(z)| dz \lesssim   \norm f_{L^1\cap L^\infty}   \ell(|x-y|),$$ 
for any $(x,y)\in \mathbb{R}^2\times \mathbb{R}^2$ with $x\neq y$.
\end{lem}

 \subsection{H{\"o}lder-continuity of some singular operators}
 For convenience, we state here several results on continuity properties of a specific type of singular operators. In particular, the ensuing bounds below will be employed  later on in the regularity analysis of   contour dynamics equation  associated with time-periodic solutions of \eqref{EQ}.

\begin{lem}[{\cite[Lemma 2.6]{HXX}}]\label{lemma:regularity-operator}
	Let $\alpha\in (0,1)$ and consider a measurable function $\mathcal{K}$ defined  on $\mathbb{T}\times \mathbb{T} \setminus \{(\theta,\theta), \theta \in \mathbb{T} \}$ with values in $\mathbb{C}$. Assume further, for all $ \theta\neq \eta \in \mathbb{T}$,  that 
	\begin{equation*}
		|\mathcal{K}(\theta,\eta) | \leq \frac{C_0 }{\left|\sin \left( \frac{\theta-\eta}{2} \right) \right|^{\alpha}  },
	\end{equation*}
	and
	\begin{equation*}
		|\mathcal{K}(\theta,\eta) | \leq \frac{C_0 }{\left|\sin \left( \frac{\theta-\eta}{2} \right) \right| ^{1+\alpha} },
	\end{equation*}
	for some $C_0>0$. Then, the integral operator  defined, for any $f\in L^\infty(\mathbb{T})$, by 
	\begin{equation*}
		\theta \mapsto \mathcal{T}f (\theta) \bydef \int_0^{2\pi}\mathcal{K}(\theta,\eta) f(\eta) d\eta
	\end{equation*}
	maps $L^\infty(\mathbb{T})$ into $C^\alpha(\mathbb{T})$. More precisely, it holds that  
	\begin{equation*}
		\norm {\mathcal{T}f}_{C^\alpha} \lesssim C_0 \norm {f}_{L^\infty}.
	\end{equation*}
\end{lem}

Later on, we are going to deal with operators having a singularity in their diagonal that corresponds to the endpoint case  $\alpha=1$  in the preceding lemma. The following lemma will then be useful in these situations.

\begin{lem}[{\cite[Proposition A.2]{GHR23}}]\label{lemma:regularity-operator2}
	Let $\alpha\in (0,1)$ and $g$ be a $C^{\alpha}(\mathbb{T})$ function. Further  consider a measurable function $\mathcal{K}$ defined  on $\mathbb{T}\times \mathbb{T} \setminus \{(\theta,\theta), \theta \in \mathbb{T} \}$ with values in $\mathbb{C}$ and assume, for all $ \theta\neq \eta \in \mathbb{T}$,  that 
	\begin{equation*}
		|\mathcal{K}(\theta,\eta) | \leq \frac{C_0 }{\left|\sin \left( \frac{\theta-\eta}{2} \right) \right|   },
	\end{equation*}
	and
	\begin{equation*}
		|\mathcal{K}(\theta,\eta) | \leq \frac{C_0 }{\left|\sin \left( \frac{\theta-\eta}{2} \right) \right| ^{2} },
	\end{equation*}
	for some $C_0>0$. Then, the integral operator  defined, for any $f\in L^\infty(\mathbb{T})$, by 
	\begin{equation*}
		\theta \mapsto \mathcal{T}_gf (\theta) \bydef \int_0^{2\pi}\mathcal{K}(\theta,\eta)\big(g(\theta)-g(\eta) \big) f(\eta) d\eta
	\end{equation*}
	maps $L^\infty(\mathbb{T})$ into $C^\alpha(\mathbb{T})$. More precisely, it holds that  
	\begin{equation*}
		\norm {\mathcal{T}_gf}_{C^\alpha} \lesssim C_0 \norm{g}_{C^\alpha} \norm {f}_{L^\infty}.
	\end{equation*}
\end{lem}

 \subsection{Crandall-Rabinowitz's theorem}
 The proof of Theorem \ref{Thm:2}, as we will show in Section \ref{section:V-states}, relies on a careful application of the generalized version of \textit{implicit functions} theorem---known as Crandall-Rabinowitz’s Theorem. The precise statement of that theorem reads as follows

		\begin{thm}[Crandall-Rabinowitz]\label{Crandall-Rabinowitz theorem}
	Let $X$ and $Y$ be two Banach spaces. Further consider  $V \subset X$  to be a neighborhood of $0$  and 
	$$F:\begin{array}[t]{rcl}
		\mathbb{R}\times V  \rightarrow   Y 
	\end{array}$$
	to be a function of class  $C^{1}$ with the following properties 
	\begin{enumerate} 
		\item (Trivial solution) For all $ \Omega\in\mathbb{R}$, we assume that $$F(\Omega,0)=0.$$
		\item (Regularity) Moreover, $F$ is assumed to be regular in the sense that  $\partial_\Omega F $, $\partial_x F $ and $ \partial_{\Omega x}^2F $ exist and are continuous.
		\item (Fredholm property) Furthermore, we assume that $\ker\left(\partial_{x}F(0,0)\right)$ is not trivial of dimension one, i.e., there is $x_0\in V$ such that
		$$\ker\left(\partial_{x}F(0,0)\right) =\langle x_{0}\rangle,$$  that  $R\left(\partial_{x}F(0,0)\right)$ is closed in $Y$  and that  $Y \setminus   R \left(\partial_{x}F(0,0)\right)$  is  one dimensional. 
		\item (Transversality assumption)  At last, we assume that  $ \partial_{\Omega x}^2F(0,0)x_{0}\not\in R\left(\partial_{x}F(0,0)\right).$
	\end{enumerate}
	Under the foregoing assumptions, for  $\chi$ being the   complement of $\ker\left(\partial_{x}F(0,0)\right)$ in $X$, there exist a neighborhood $U$ of $(0,0)$, an interval $(-a,a)$, for some $a>0$ and    continuous functions
	$$\psi:(-a,a)\rightarrow \mathbb{R}\quad\textnormal{and}\quad\phi:(-a,a)\rightarrow\chi$$
	such that $$\psi(0)=  \phi(0)=0$$ 
	and 
	$$\Big\{(\Omega,x)\in U: \; F(\Omega,x)=0\Big\}=\Big\{\big(\psi(s),sx_{0}+s\phi(s)\big): \;|s|<a\Big\}\cup\Big\{(\Omega,0)\in U\Big\}.$$
\end{thm}

 \section{Weak and Lagrangian solutions}\label{section:Yudovich}
 This section is devoted to the proof of Theorem \ref{Thm:1}. We proceed first by introducing several notation that will be employed in the proof below. For a given $\delta>0$, we introduce the matrix
  \begin{equation}\label{A-delta:def}
\mathcal{A}_\delta \bydef    \begin{pmatrix} 
		 1 & \delta ^{-1}\\
		 1 & -1
	\end{pmatrix} 
\end{equation} 
and, for a couple of real valued  functions  $f=(f_1,f_2) $, we define
\begin{equation*} 
	\begin{pmatrix}
		f_+\\f_- 
	\end{pmatrix} \bydef   
	\mathcal{A}_{\delta}
	 \cdot
	\begin{pmatrix}
	f_1 \\ f_2
	\end{pmatrix}.
\end{equation*}
Note that
\begin{equation*} 
	\begin{pmatrix}
		f_1\\f_2 
	\end{pmatrix} = \mathcal{A}^{-1}_\delta
	 \cdot
	\begin{pmatrix}
	f_+ \\ f_-
	\end{pmatrix},
\end{equation*} 
where $\mathcal{A}_\delta^{-1}$ is the inverse of $\mathcal{A}_\delta$, which is given by 
$$\mathcal{A}^{-1}_\delta =  \left(1+ \delta^{-1}\right)^{-1}
	A_\delta.$$
Moreover, for another couple of real valued functions $ \widetilde{f}=(\widetilde{f}_1,\widetilde{f}_2)$,  the following elementary inequality can justified by a direct computation
\begin{equation*} 
\begin{aligned}
\frac{1}{ 2 \max\{ 1, \delta^{-1} \}} &\left(   \textbf{d}( f_+ , \widetilde{f} _+) + \textbf{d}( f_- , \widetilde{f} _-) \right) \\
& \qquad \qquad    \leq  \;\textbf{d}( f_1 , \widetilde{f} _1) + \textbf{d}( f_2 , \widetilde{f} _2)\\
&\qquad  \qquad  \qquad  \qquad  \qquad    \leq  \frac{1 + \max\{ 1, \delta^{-1} \}}{1+ \delta^{-1}} \left( \textbf{d}( f_+ , \widetilde{f} _+) + \textbf{d}( f_- , \widetilde{f} _-) \right), 
\end{aligned}   
\end{equation*}   
where $\textbf{d}(A,B)$ denotes the distance between $A$ and $B$.   For simplicity, we shall use the following version of the preceding inequalities
\begin{equation}\label{inequa:2}
\begin{aligned}
C_{\delta}^{-1} \left(   \textbf{d}( f_+ , \widetilde{f} _+) + \textbf{d}( f_- , \widetilde{f} _-) \right) \leq   \;\textbf{d}( f_1 , \widetilde{f} _1) + \textbf{d}( f_2 , \widetilde{f} _2)  \leq C_{\delta} \left( \textbf{d}( f_+ , \widetilde{f} _+) + \textbf{d}( f_- , \widetilde{f} _-) \right), 
\end{aligned}   
\end{equation}   
where $C_{\delta}\geq 1$ is a constant that only depends  on $\delta.$
\subsection{Existence of global solutions}
Here, we show existence of global solutions of \eqref{EQ} as is claimed in Theorem \ref{Thm:1}.  

\subsubsection{Existence of weak-distributional solutions} Weak solutions can be constructed by analyzing the convergence of the scheme below associated with a  linearized transport equations. To that end, we assume that we have constructed a velocity $u^n_{i}$, for $n\in \mathbb{N}$, then   build up the next iteration $(\omega^{n+1},u^{n+1})$ by solving the linear system of equations
\begin{equation*} 
\left\{
\begin{array}{ll}
\partial_t \omega_i^{n+1} + v_i^{n } \cdot \nabla \omega_i^{n+1}=0,\\ \vspace{1mm} \\
\begin{pmatrix}
		\omega _+^{n+1}\\
		  \omega _-^{n+1}
	\end{pmatrix} \bydef    \mathcal{A}_{\delta}  	\cdot\begin{pmatrix}
		 \omega _1 ^{n+1}\\
		  \omega _2^{n+1}
	\end{pmatrix}
 , \\ ~~\\
   v_+^{n+1} \bydef    K_+ \omega_+^{n+1} , \qquad v_-^{n+1}  \bydef    K_- \omega_-^{n+1} ,\\ \vspace{1mm} \\
 	\begin{pmatrix}
		 v_1^{n+1}\\
		  v_2^{n+1}
	\end{pmatrix} \bydef    \mathcal{A}_{\delta}^{-1} 	\cdot\begin{pmatrix}
		 v_+^{n+1}\\
		  v_-^{n+1}
	\end{pmatrix},
\end{array}
\right.
\end{equation*}
supplemented with the initial data
$$ \omega_i^{n+1}{|_{t=0}} = \omega_i(0), \quad \text{for all} \quad n\in \mathbb{N},$$
 and the initial iteration
  $$
\begin{pmatrix}
		 v_1^{0}\\~~\\
		  v_2^{0}
	\end{pmatrix} \bydef     \mathcal{A}_{\delta}^{-1} \cdot \begin{pmatrix}
		 v_+^{0}\\~~\\
		  v_-^{0}
	\end{pmatrix} =  \mathcal{A}_{\delta}^{-1} \cdot \begin{pmatrix}
		 K_+ \omega_+(0)\\~~\\
		  K_-\omega_- (0)
	\end{pmatrix} =  \mathcal{A}_{\delta}^{-1} \cdot \begin{pmatrix}
		 K_+ \left(\mathcal{A}_\delta \cdot   \begin{pmatrix}
		 \omega _1 (0)\\
		  \omega _2(0)
	\end{pmatrix}\right)\\ ~~\\
		  K_- \left(\mathcal{A}_\delta \cdot   \begin{pmatrix}
		 \omega _1 (0)\\
		  \omega _2(0)
	\end{pmatrix}\right)
	\end{pmatrix}.
 $$
The adequate bounds can be proved by induction and by performing   standard energy estimates on  transport equation, altogether with the employment of the bounds on the operators $K_{\pm}$ given by Lemma \ref{kernel:ES}. This process establishes the existence of bounded sequence  of solutions $ (\omega_1^n, \omega_2^n)_{n\in\mathbb{N}}$ that convergences, in the sense of distributions, to a weak solution of the system of equations \eqref{EQ}. The compactness of the approximate solution $ (\omega_1^n, \omega_2^n)_{n\in\mathbb{N}}$ and the convergence in nonlinear terms can be done by using classical techniques. Finally, we emphasize that the resulting limit as $n\to \infty$ enjoys the following  bounds 
$$ \omega_1,  \omega_2\in C(\mathbb{R}^+, L^p(\mathbb{R}^2)) \cap C_w(\mathbb{R}^+, L^\infty(\mathbb{R}^2)) ,  $$
$$v_1, v_2 \in C(\mathbb{R}^+, L^q(\mathbb{R}^2)) \cap L^\infty(\mathbb{R}^+, LL(\mathbb{R}^2)) ,$$
for all $p\in [1,\infty)$ and   $q\in (2,\infty].$

\subsubsection{Existence of Lagrangian solutions} Now, we show that any weak solution is Lagrangian. To that end, for all $i\in\{1,2\}$, we first define the flow-maps $X_i$  associated with the  velocity field $v_i$ as  the unique solution of  \eqref{ODE}. Again, we emphasize that  the flow-map $X_i$ is uniquely determined due to the Log-lipschitz regularity of $v_i$ (see \cite[Section 5.2]{CH95}). Therefore, we define the Lagrangian solution $ \Omega_i$ of \eqref{EQ} by the push forward 
$$\Omega_i \bydef     \omega_{i,0}\left(X_i^{-1}(t,\cdot) \right) ,\quad \text{for all}\quad t\geq 0 .$$
Observe that $ \Omega_i $ satisfies the same transport equation   as $ \omega_i$ (with velocity $v_i$), i.e.,
$$ \partial_t \Omega_i + v_i \cdot \nabla \Omega_i =0.$$ Hence, by classical results on transport equations (see for instance \cite{DL}) we deduce that $\Omega_i = \omega_i$, thereby establishing that $\omega_i$ is a Lagrangian solution.

\subsection{Stability of the flow-maps and uniqueness} The proof of the uniqueness follows     the idea from \cite{CRIPPA1} and will be acheived by establishing a stability result for the flow-maps associated with two different solutions. 

In the sequel, $C_0>0$ denotes  a constant that depends only on   the $L^1\cap L^\infty(\mathbb{R}^2)$ norm  of the initial data and is allowed to differ from one line to the next one.

Let $(v_i,\omega_i)$ and $(\widetilde{v}_i,\widetilde{\omega}_i)$ be two Lagrangian solutions to \eqref{EQ} with the same initial data $\omega_{i,0}$ belonging to $ L^1\cap L^\infty (\mathbb{R}^2)$. We denote by $X_i$ and $\widetilde{X}_i$ the unique flows associated to each solution, for $i\in \{1,2\}$. Further introduce 
$$
\begin{pmatrix}
		X_+\\~~\\
		X_-
	\end{pmatrix} \bydef     \mathcal{A}_{\delta}  \cdot \begin{pmatrix}
		X_1\\~~\\
		 X_2
	\end{pmatrix} , \qquad 
	\begin{pmatrix}
		\widetilde{X}_+\\~~\\
		\widetilde{X}_-
	\end{pmatrix} \bydef     \mathcal{A}_{\delta}  \cdot \begin{pmatrix}
		\widetilde{X}_1\\~~\\
		 \widetilde{X}_2
	\end{pmatrix},
	$$
	where $A_{\delta}$ is the matrix defined in \eqref{A-delta:def}. Then,    in view of the equation \eqref{ODE} satisfied by each  flow-map, we compute that
 \begin{equation*}
 \begin{aligned}
 X_+(t) - \widetilde{X}_+(t) & = \int_0^t \left(v_1(X_1(s)) + \delta^{-1} v_2 (X_2(s)) - \widetilde{v}_1(\widetilde{X}_1(s)) - \delta^{-1}\widetilde{v}_2 (\widetilde{X}_2(s)) \right) ds \\
  &= \int_0^t \left( v_1 (X_1(s)) - v_1(\widetilde{X}_1(s)) + (v_1 - \widetilde{v} _1)(\widetilde{X}_1(s)) \right) ds \\
&\quad +  \delta^{-1}  \int_0^t \left( v_2 (X_2(s)) - v_2(\widetilde{X}_2(s)) + (v_2 - \widetilde{v} _2)(\widetilde{X}_2(s)) \right) ds ,
 \end{aligned}
 \end{equation*}
 for all $t\geq 0$. Therefore, by further introducing the fluctuation
  $$  X \bydef    |X_1- \widetilde{X}_1| + |X_2-\widetilde{X}_2|,$$
 and  owing to the fact that $ v_1$ and $v_2$ are Log-Lipschitz, altogether with \eqref{inequa:2}, we infer that
 \begin{equation*}
 \begin{aligned}
 \big|X_+(t) - \widetilde{X}_+(t)\big| & \lesssim_ {\delta} C_0 \int_0^t \ell(| X_1-\widetilde{X}_1(s)|) + \ell(| X_2-\widetilde{X}_2(s)|)    ds \\
 &\qquad+ \int_0^t \big|( v_+ -\widetilde{v}_+)(\widetilde{X}_1(s)) \big|+ \int_0^t \big|( v_- -\widetilde{v}_-)(\widetilde{X}_1(s)) \big|ds\\
 &\qquad+ \int_0^t \big|( v_+ -\widetilde{v}_+)(\widetilde{X}_2(s)) \big|+ \int_0^t \big|( v_- -\widetilde{v}_-)(\widetilde{X}_2(s)) \big|ds\\
 & \lesssim_\delta C_0 \int_0^t \ell(  X(s))    ds \\
 & \quad +\sum_{i=1}^2 \left( \int_0^t \big|( v_+ -\widetilde{v}_+)(\widetilde{X}_i(s)) \big|  ds + \int_0^t \big|( v_- -\widetilde{v}_-)(\widetilde{X}_i(s)) \big|  ds \right),
 \end{aligned}
 \end{equation*} 
where  the function $\ell$ is defined in \eqref{ell:def}.  
Moreover, in the same way, we obtain that
 \begin{equation*}
 \begin{aligned}
 \big|X_-(t) - \widetilde{X}_-(t)\big|    
 & \lesssim_{\delta} C_0 \int_0^t \ell(  X(s))    ds  \\
 & \quad +\sum_{i=1}^2 \left( \int_0^t \big|( v_+ -\widetilde{v}_+)(\widetilde{X}_i(s)) \big|  ds + \int_0^t \big|( v_- -\widetilde{v}_-)(\widetilde{X}_i(s)) \big|  ds \right).
 \end{aligned}
 \end{equation*}
Hence, utilizing \eqref{inequa:2}, again, entails that
  \begin{equation*}
 \begin{aligned}
  X(t)  
 & \lesssim_\delta C_0 \int_0^t \ell(  X(s))    ds +\sum_{i=1}^2 \left( \int_0^t \big|( v_+ -\widetilde{v}_+)(\widetilde{X}_i(s)) \big|  ds + \int_0^t \big|( v_- -\widetilde{v}_-)(\widetilde{X}_i(s)) \big|  ds \right).
 \end{aligned}
 \end{equation*}
 On the other hand,   observing, for all $i\in \{1,2\}$, that    
  \begin{equation*}
 \begin{aligned}
 \big|( v_+ -\widetilde{v}_+)(\widetilde{X}_i) \big|  & = \big|(K _+\omega_+(\widetilde{X}_i) - K_+\widetilde{\omega} _+(\widetilde{X}_i) \big| \\
 &\lesssim_{\delta} \sum_{j=1}^2 \big|(K _+\omega_j(\widetilde{X}_i) - K_+\widetilde{\omega} _j(\widetilde{X}_i) \big|   ,
 \end{aligned}
 \end{equation*}
  and   exploiting the definition of   $K_+$ allows us to write, for all $i,j\in \{1,2\}$, that 
   \begin{equation*}
 \begin{aligned}
  \big|(K_+ \omega_j(\widetilde{X}_i) - K_+\widetilde{\omega} _j(\widetilde{X}_i) \big|  &= \Big|  \int_{\mathbb{R}^2} k_+(\widetilde{X}_i(s,x),y) \omega_j(y) dy - k_+(\widetilde{X}_i(s,x),y) \widetilde{\omega} _j(y) dy  \Big| \\
  &= \Big|  \int_{\mathbb{R}^2} k_+(\widetilde{X} _i(s,x),X _j(s,y) ) \omega_{j,0}(y)  \\
  & \qquad \qquad - k_+(\widetilde{X}_i(s,x),\widetilde{X} _j(s,y)) \omega  _{0,j}(y) dy  \Big| \\
  &\leq    \int_{\mathbb{R}^2}\Big| k_+(\widetilde{X} _i(s,x),X _j(s,y) )   \\
  & \qquad \qquad- k_+(\widetilde{X}_i(s,x),\widetilde{X} _j(s,y))\Big| \Big|  \omega  _{0,j}(y)\Big| dy.  
 \end{aligned}
 \end{equation*}
 Consequently, we find that 
 \begin{equation*}
 \begin{aligned}
  \big|( v_+-\widetilde{v}_+)(\widetilde{X}_i) \big|  &  
  &\lesssim_\delta  \sum_{j=1}^2   \int_{\mathbb{R}^2}\Big| k_+(\widetilde{X} _i(s,x),X _j(s,y) )   -  k_+(\widetilde{X}_i(s,x),\widetilde{X} _j(s,y))\Big| \Big|  \omega  _{0,j}(y)\Big| dy.  
 \end{aligned}
 \end{equation*}
 Likewise, we obtain for  the minus parts
  \begin{equation*}
 \begin{aligned}
  \big|( v_- -\widetilde{v}_-)(\widetilde{X}_i) \big|  & = \big|(K_- \omega_-(\widetilde{X}_i) - K_-\widetilde{\omega} _-(\widetilde{X}_i) \big| \\
  &\lesssim_\delta \sum_{j=1}^2 \big|(K_- \omega_j(\widetilde{X}_i) - K_-\widetilde{\omega} _j(\widetilde{X}_i) \big| \\
  &\lesssim_\delta  \sum_{j=1}^2   \int_{\mathbb{R}^2}\Big| k_-(\widetilde{X} _i(s,x),X _j(s,y) )     -  k_-(\widetilde{X}_i(s,x),\widetilde{X} _j(s,y))\Big| \Big|  \omega  _{0,j}(y)\Big| dy.  
 \end{aligned}
 \end{equation*} 
All in all, gathering the foregoing estimates   yields
  \begin{equation}\label{pre-last:X}
 \begin{aligned}
   X (t)   
 & \lesssim_{\delta} C_0 \int_0^t \ell( X(s))    ds \\
 &\quad +\sum_{\overset{i,j=1}{m = \pm}}^2   \int_0^t \int_{\mathbb{R}^2}\Big| k_m(\widetilde{X} _i(s,x),X _j(s,y) )    -  k_m(\widetilde{X}_i(s,x),\widetilde{X} _j(s,y))\Big| \Big|  \omega  _{0,j}(y)\Big| dy ds.   
 \end{aligned}
 \end{equation}
Onwards, we   omit the subscript $\delta$ and we keep in mind that the constants in the   estimates  below depend on $\delta,$ as well. 
Let $\alpha>0$ be a function in $L^1\cap L^\infty(\mathbb{R}^2) $. Accordingly, we define the density  $$\eta  \bydef    \sum_{ i=1}^2 |\omega_{i,0}| + \alpha$$   and its push-forward by the mapping $ \widetilde{X}_j(s,\cdot)$  
   $$ \eta_j(s,\cdot)\bydef     \eta( {\widetilde{X}_j }^{-1} (s,\cdot)).$$ 
 Then,  we proceed by integrating \eqref{pre-last:X} against   to the measure $ \eta(x)dx$ and   utilizing Fubini's theorem  to obtain that
    \begin{equation*}
 \begin{aligned}
&\int_{\mathbb{R}^2}  X   (t,x) \eta(x)dx
  \lesssim C_0 \int_0^t  \int_{\mathbb{R}^2}\ell\big( X(s,x)\big)  \eta(x)dx ds \\
 &\quad \quad +\sum_{\overset{i,j=1}{m = \pm}}^2     \int_0^t \int_{\mathbb{R}^4}\Big| k_m(\widetilde{X} _i(s,x),X _j(s,y) )   -  k_m(\widetilde{X}_i(s,x),\widetilde{X} _j(s,y))\Big| \Big|  \omega  _{0,j}(y)\Big| dy \eta(x)dx ds\\
 &  \quad = C_0 \int_0^t  \int_{\mathbb{R}^2}\ell\big( X(s,x)\big)   \eta(x)dx  ds  \\
 & \quad \quad+\sum_{\overset{i,j=1}{m = \pm}}^2     \int_0^t \int_{\mathbb{R}^2}  \Big|  \omega  _{0,j}(y)\Big|     \int_{\mathbb{R}^2}\Big| k_m(\widetilde{X} _i(s,x),X _j(s,y) ) -  k_m(\widetilde{X}_i(s,x),\widetilde{X} _j(s,y))\Big|  \eta(x)dx   dy  ds.
 \end{aligned}
 \end{equation*}
Now,  owing to the fact that $\widetilde{X}_i(s,\cdot)$ preserves volumes, for any $i\in\{ 1,2\}$ and $s\geq 0$, and appealing to Lemma \ref{kernel:ES} entails, for any $i,j\in \{1,2\}$ and $m = \pm $, that
   \begin{equation*}
 \begin{aligned}
 \int_{\mathbb{R}^2}\Big| k_\pm(\widetilde{X} _i(s,x),X _j(s,y) )     &-  k_\pm(\widetilde{X}_i(s,x),\widetilde{X} _j(s,y))\Big|  \eta(x)dx  \\
 &=   \int_{\mathbb{R}^2}\Big| k_\pm(x,X _j(s,y) )      -  k_\pm(x,\widetilde{X} _j(s,y))\Big|  \eta_j (  s,x )dx \\ 
 & \lesssim \norm { \eta_j (  s,\cdot ) }_{L^1\cap L^\infty} \ell\big( |X _j(s,y) - \widetilde{X} _j(s,y) |\big)\\
 & \lesssim C_0\ell\big(     X (s,y)\big).
 \end{aligned}
 \end{equation*}
 Therefore, it follows that
   \begin{equation*}
 \begin{aligned}
\int_{\mathbb{R}^2}   X   (t,x)  \eta(x)dx
  &\lesssim C_0 \int_0^t  \int_{\mathbb{R}^2}\ell( X(s,x)) \eta(x)dx ds  \\
  &\qquad+ C_0\sum_{j=1}^2     \int_0^t \int_{\mathbb{R}^2}  \Big|  \omega  _{0,j}(y)\Big|\ell( X(s,y))      dy  ds\\
  &\lesssim C_0 \int_0^t  \int_{\mathbb{R}^2}\ell(  X(s,x))  \eta(x)dx ds   .
 \end{aligned}
 \end{equation*}
 At last, since the  function $\ell$ is concave and $\eta \in L^1(\mathbb{R}^2)$, we end up with
 $$\int_{\mathbb{R}^2}  X   (t,x)  \eta(x)dx \lesssim C_0 \int_0^t \ell\left( \int_{\mathbb{R}^2}  X   (s,x)  \eta(x)dx \right)  ds.$$ 
 The uniqueness of the flow-maps follows due to Osgood's Lemma, see  \cite[Lemma 3.4]{bahouri2011fourier}, for instance. Consequently, we deduce that $\omega_i \equiv \widetilde{\omega}_i $, thereby showing the uniqueness of the solution and completing the proof of  Theorem \ref{Thm:1}.  \qed

  \section{Uniformly rotating solutions}\label{section:V-states}
  This section is devoted to the proof of    existence of rotating $m$-fold vortex patches (relative) equilibriums for the Surface Quasi-Geostrophic  model with two boundaries   given by \eqref{EQ}. The statement of that is given by Theorem \ref{Thm:2-A} which we recast below in a more precise format. To that end, allow us first to set up few notations that will be constantly used throughout this section. For a given integer $m\in \mathbb{N}^*$, and real numbers $0<b_2\leq b_1<\infty$, we introduce the set 
  $$ {S}_{m,b_1}\bydef    \left\{ s\in (0,b_1]: \exists n\in\mathbb{N}^*,\ \Omega_{m}^-(s,b_1) =   \Omega_n^+(s,b_1) \right\}$$ 
  where
  
  \begin{equation*} 
\Omega_n^{\pm}(b_1,b_2)\bydef \frac{1}{2(\delta + 1)} \left(  - (A_{b_1,b_2,n}+B_{b_1,b_2,n})\pm  \sqrt{(A_{b_1,b_2,n}-B_{b_1,b_2,n})^2 +{4\delta}\gamma_{b_1,b_2,n} ^2} \right),
\end{equation*}
where, we set 
\begin{equation*} 
     \begin{aligned}
	A_{b_1,b_2,n}&\bydef   (\delta+1)V_{b_1,b_2}+\frac{\delta}{2n}+I_{n}(b_1\mu)K_{n}(b_1\mu) 
	\\
	B_{b_1,b_2,n}&\bydef   (\delta+1)W_{b_1,b_2}+\frac{1}{2n}+{\delta} I_{n}(b_2\mu)K_{n}(b_2\mu),
	\\
	\gamma_{b_1,b_2,n} & \bydef  \ \frac{b^{n}}{2n}-I_{n}(b_2\mu )K_{n}(b_1\mu) , 
\end{aligned}
\end{equation*} 
  with $b\bydef \frac{b_2}{b_1}$ and
  \begin{equation}\label{V120}
      		\begin{aligned}  
	 V_{b_1,b_2} & \bydef-\frac{\delta+b^2}{2(1+\delta)}-\frac{1}{1+\delta}\Big(I_{1}(b_1\mu)K_{1}(b_1\mu)-b I_{1}(b_2 \mu)K_{1}(b_1\mu)\Big) ,
	 \\
	 W_{b_1,b_2}  &\bydef  -\frac12-\frac{\delta}{1+\delta}\Big(I_{1}(b_2\mu)K_{1}(b_2\mu)-\frac 1b I_{1}(b_2 \mu)K_{1}(b_1 \mu)\Big).
	\end{aligned}
	\end{equation}
	Above and in what follows, we denote
$$
  \mu\bydef   \lambda\sqrt{1+\delta},
$$
where $\Lambda, \delta>0$ are the constants appearing in the system of equations \eqref{EQ}, whereas  we recall that $I_n$ and $K_n$, for $n\in \mathbb{N}^*$, stand for the modified Bessel functions previously introduced in Section \ref{section:Bessel}.

For a mere simplification of notation, we drop the dependence of $\Omega_n^\pm(b_1,b_2)$, as well as $A_{b_1,b_2,n}$ and $B_{b_1,b_2,n}$, on $b_1$ and $b_2$ even though this is crucial and will be utilized in our proofs, later on. 

Theorem \ref{Thm:2-A} is a consequence of the following slightly more precise theorem.
  
  \begin{thm}[Periodic solutions bifurcating from simple eigenvalues]\label{Thm:2}
  Let $m\in\mathbb{N}^*$ and  $\lambda >0$. Further fix two radii $0<b_2\leq b_1$ and assume that $ \delta \geq b^2$, where $b=\frac{b_2}{b_1}$. Then, the set $S_{m,b_1}$ defined above contains at most a finite number of elements. Moreover,  the following holds:
  \begin{itemize}
  	\item Assuming  that $b_2$ does not belong to $  S_{m,b_1}$, there exist  two curves of $m$-fold symmetric pairs of   simply connected V-states, for any $m\in \mathbb{N}^*$, bifurcating from the steady states 
  	\begin{equation*} 
  \omega_1= \mathds{1}_{\{x\in \mathbb{R}^2 :  |x| < b_1\}} \quad \text{and} \quad \omega_2= \mathds{1}_{\{x\in \mathbb{R}^2 :  |x| < b_2\}}
  \end{equation*} 
  	  at each of the angular velocities $\Omega=\Omega_m^\pm$ defined above.
    \item Assuming $b_1\neq b_2$, there is $m_0\in \mathbb{N}^*$ such that  $S_{m,b_1}$ is empty for any $m\geq m_0$ and for which there exist two curves of $m$-fold symmetric pairs of   simply connected V-states satisfying the same properties as in the preceding case. 
  \end{itemize}  
 \end{thm}
  
  \begin{rem}
  	In the case of identical initial discs, we will show in the proof of Proposition \ref{prop.last} below that  there are   values $b_1=b_2 \in (0,\infty)$ for which spectral collisions can happen, i.e.      situations where the angular velocities satisfy $\Omega_m^-=\Omega_n^+$, for some $n,m\in \mathbb{N}^*$. In that particular case, we are not able to directly apply Crandall-Rabinowitz's Theorem \ref{Crandall-Rabinowitz theorem} to prove existence of time-periodic solutions bifurcating from the same initial disc. This also shows that the set $ S_{m,b_1}$ is not empty in general.
  	  \end{rem}
  
   \begin{rem}\label{RMK:bi}
  	Note that Theorem \ref{Thm:2} covers Theorem \ref{Thm:2-A}. Moreover, Theorem \ref{Thm:2} emphasizes that $\delta$ is allowed to take values in $[b^2,1]$, as well. This condition can probably be further relaxed so that $\delta$ would take values in the whole half-line $(0,\infty)$, though a justification of that is not available   in our proof below. Indeed, it is to be emphasized later on that  the sequences of  angular velocities  $ (\Omega_n^\pm )_{n\in \mathbb{N}^*}$ are non-decreasing for any value of $\delta$ in $(0,\infty)$, which means that spectral collisions cannot occur from the same sequence $ (\Omega_n^+ )_{n\in \mathbb{N}^*}$ or $(\Omega_n^- )_{n\in \mathbb{N}^*}$. Furthermore, under the assumption that $ \delta\geq b^2$ with $b_1 \neq b_2$, it is shown in the proof of Proposition \ref{lem-spec} below that these sequences have different limits at infinity. This allows to avoid spectral collusions between the sequences $ (\Omega_n^+ )_{n\in \mathbb{N}^*}$ and $(\Omega_n^- )_{n\in \mathbb{N}^*}$ for large symmetries. However, because of the implicit expression of the Bessel functions, we were not able to rigorously  justify that fact in the range of parameter $\delta \in (0,b^2)$, even though   several numerical simulations, that we do not include in this paper, suggest that Theorem \ref{Thm:2} would hold for any value of $\delta>0$. 
  \end{rem}

  The proof of Theorem \ref{Thm:2} is carried out in several steps and is detailed in the forthcoming sections. In particular, our strategy consists in   establishing all the prerequisites  in Crandall-Rabinowitz's Theorem \ref{Crandall-Rabinowitz theorem} which, subsequently, will grant us   the results of Theorem \ref{Thm:2}.

    \subsection{Contour dynamics equation}
 Here, we set up the contour dynamics   governing the   motion of   vortex patches.
Onward, we identify $\mathbb{C}$ with $\mathbb{R}^{2},$ where  $\mathbb{C}$ is naturally endowed with the Euclidean scalar product which writes, for $z_{1}=x_{1}+\ii y_{1}$ and $z_{2}=x_{2}+\ii y_{2}$, as
	$$ z_{1}\cdot z_{2}\bydef   \mbox{Re}(\overline{z_{1}}z_{2})=\tfrac{1}{2}\left(\overline{z_{1}}z_{2}+z_{1}\overline{z_{2}}\right)=x_{1}x_{2}+y_{1}y_{2}.$$
	
 Let $D_1$ and $D_2$ be two simply-connected domains,   close to the discs, of radii $b_1$ and $b_2$, respectively. Further consider polar parametrization of the associated boundaries, for $ k\in\{1,2\}$
\begin{equation}\label{def zk}
  z_{k}:\begin{array}[t]{rcl}
		\mathbb{T} & \mapsto & \partial D_{k}\\
		\theta & \mapsto & R_k(\theta)e^{\ii\theta}\bydef   \sqrt{b_k^{2}+2r_k(\theta) } e^{\ii\theta},
	\end{array}
	\end{equation} 
	for some function $r(\theta).$ Accordingly, we introduce  the initial vorticities
$$
(\omega_1,\omega_2)|_{t=0}=\big(\mathds{1}_{D_{1}},\mathds{1}_{D_{2}}\big)
$$
and assume that $\omega_1$ and $\omega_2$ give rise to two rotating patch solutions  of \eqref{EQ} about  the origin, with an angular velocity $\Omega\in \mathbb{R}$.  
 
 The vortex patch equation \eqref{vortex patch equation} provides a system of coupled nonlinear and nonlocal PDE satisfied by the radial deformations $r_1$ and $r_2$. The precise statement is the content of the next lemma.

	\begin{lem}\label{lem boundary eq} 
		The radial deformations $r\bydef (r_1,r_2)$  defined through \eqref{def zk} satisfy the nonlinear coupled system
		\begin{equation}\label{def F}
					\mathcal{F}(\Omega,r)\bydef   \big(\mathcal{F}_1(\Omega,r),\mathcal{F}_2(\Omega,r)\big)=0,  
		\end{equation}
		where we set, for all $k,j\in \{1,2\}$ and $\theta\in \mathbb{T}$, that 
		\begin{equation}\label{Fk:def}
			\begin{aligned}
				\mathcal{F}_k(\Omega,r)(\theta) & \bydef    \Omega \partial_\theta r_k(\theta)+\mathcal{F}_{k,1} ( r)(\theta)+\mathcal{F}_{k,2}( r)(\theta),
			\end{aligned}
		\end{equation}
		and
		\begin{equation}\label{def fkj}
			\mathcal{F}_{k,j}( r)(\theta) \bydef   \int_{0}^{2\pi} G_{k,j}\big(R_k(\theta)e^{\ii \theta}-R_j(\eta) e^{\ii \eta}\big) \partial^2_{\theta \eta}\Big(R_k(\theta) R_j(\eta)\sin(\eta-\theta) \Big)d\eta, 
		\end{equation} 
		where the kernels $G_{k,j}$ are introduced in \eqref{def Gkj}. 
		
		Moreover, the Rankine vorticities, corresponding to  $r\equiv 0$, are stationary solutions, for any radii $b_1,b_2>0$. More precisely, one has  that
		$$
		   \mathcal{F}(\Omega,0)=0,
		$$
		for all  $ \Omega\in \mathbb{R}$.
	\end{lem}

\begin{proof}	
Note first that    the normal inward vector to the boundary $\partial D_{k}$  of the patch at the point $z_k(\theta)$ is given by 
$$n_k( z_k(\theta)) =\ii \partial_\theta z_k(\theta) ,$$
for any $k\in \{ 1,2\}$.
 According to \eqref{vortex patch equation}, the initial datum  $\omega_k|_{t=0}=\mathds{1}_{D_{k}}$ generates a rotating patch about the origin with uniform angular velocity  $\Omega\in\mathbb{R}$ if and only if it holds, for all $\theta\in \mathbb{T}$, that
\begin{align}\label{boundary eq in}
 \Omega\,\textnormal{Re}\Big\{ \overline{{z}_k(\theta)}\partial_{\theta}{z}_k(\theta)\Big\}=2\textnormal{Re}\Big\{\partial_{\overline z} \psi_k({z}_k(\theta))\partial_\theta \overline{ {z}_k(\theta)}\Big\}.
\end{align}
Therefore,  observing from \eqref{def zk}   that
\begin{align*}
\Omega\partial_\theta r_k(\theta)= \Omega\,\textnormal{Re}\Big\{ \overline{{z}_k(\theta)}\partial_{\theta}{z}_k(\theta)\Big\},
\end{align*}
entails, for all $k\in \{1,2\}$ and $\theta\in \mathbb{T}$, that
\begin{align}\label{VPS}
  \Omega\partial_\theta r_k(\theta)=2\textnormal{Re}\Big\{\partial_{\overline z} \psi_k({z}_k(\theta))\partial_\theta \overline{ {z}_k(\theta)}\Big\}.
\end{align}
Hence, writing, by the chain rule,  that
\begin{align}\label{boundary eq in2}
2\textnormal{Re}\Big\{\partial_{\overline z} \psi_k({z}_k(\theta))\partial_\theta \overline{ {z}_k(\theta)}\Big\}=\partial_\theta\big( \psi_k(t,{z}_k(\theta))\big),
\end{align}
leads, for all $k\in \{1,2\}$ and $\theta\in \mathbb{T}$, to the equivalence reformulation
\begin{equation*}
	\Omega \partial_\theta r_k(\theta)=\partial_{\theta}\big( \psi(z_k(\theta))\big).
\end{equation*} 
Now, we recall from \eqref{def streamL1} that         \begin{align*}
	\psi_k(z)&=\sum_{j=1}^2\int_{D_j}G_{k,j}(z-\xi)d \xi,
	\end{align*}
for all $z\in \mathbb{C}$, where $G_{k,j}$ is given by \eqref{def Gkj}.
	Observe that the preceding representation can be recast in   polar coordinates \eqref{def zk} as
	\begin{align}\label{def stream polar}
	\psi_k(z)&=\sum_{j=1}^2\int_{0}^{2\pi}\int_0^{R_j(\eta)}G_{k,j}(z-\rho e^{\ii \eta})\rho d\rho d\eta .
		\end{align}
Therefore,  employing the identity
$$\partial_{\overline z} G_{k,j}(z,\xi)=-\partial_{\overline \xi} G_{k,j}(z-\xi),$$ 
yields, for any $z\in \mathbb{C}$, that
	\begin{align*}
	\partial_{\overline z} \psi_k\big(z\big)&=-\sum_{j=1}^2\int_{0}^{2\pi}\int_0^{R_j(\eta)}\partial_{\overline\xi} G_{k,j}\big(z-\rho e^{\ii \eta}\big) \rho  d\rho d\eta.
		\end{align*}
Thus, by further employing  Gauss--Green theorem, we arrive at the representation
	\begin{align}\label{d z psi}
		2\partial_{\overline z} \psi_k\big(z\big)&=\ii\sum_{j=1}^2 \int_{0}^{2\pi}G_{k,1}\big(z-R_j(\eta) e^{\ii \eta}\big)\partial_\eta\big(R_j(\eta)e^{\ii \eta} \big)\big)d\eta,
		\end{align}
whereby we deduce, for any $k\in \{1,2\}$ and $\theta\in \mathbb{T}$, that
	\begin{align*}
&	2\textnormal{Re}\Big\{\partial_{\overline z} \psi_k\big({z}_k(\theta)\big)\partial_\theta\overline{{z}_k(\theta)}\Big\}\\ &\quad \qquad = -\sum_{j=1}^2\int_{0}^{2\pi}G_{k,j}\big(R_k(\theta)e^{\ii \theta}-R_j(\eta) e^{\ii \eta}\big) \textnormal{Im}\Big\{ \partial_\theta\big(R_k(\theta)e^{-\ii \theta} \big)\partial_\eta\big(R_j(\eta)e^{\ii \eta} \big)\Big\}d\eta.
		\end{align*}
		At last, inserting the latter identity above in \eqref{VPS} and employing the fact  that
		\begin{equation*}
			\textnormal{Im}\Big\{ \partial_\theta\big(f(\theta)e^{-\ii \theta} \big)\partial_\eta\big(g(\eta)e^{\ii \eta} \big)\Big\} = \partial^2_{\theta \eta}\Big(f(\theta) g(\eta)\sin(\eta-\theta) \Big),
		\end{equation*}
		for any functions $f,g$, yields the contour dynamics equation.
		
		 To conclude, by a direct computation, one sees that		
		 	\begin{align*}
\mathcal{F}_k(\Omega,0)(\theta)&=b_k  \sum_{j=1}^2 b_j\int_{0}^{2\pi}G_{k,j}\big(b_k e^{\ii \theta}-b_j e^{\ii \eta}\big) \sin(\eta-\theta)  d\eta
\\ &=b_k  \sum_{j=1}^2 b_j\int_{0}^{2\pi}G_{k,j}\big(b_k -b_j e^{\ii (\eta-\theta)}\big) \sin(\eta-\theta)  d\eta
\\ &=b_k \sum_{j=1}^2 b_j\int_{0}^{2\pi}G_{k,j}\big(b_k -b_j e^{\ii \eta}\big) \sin(\eta)  d\eta,\nonumber
		\end{align*}
		for any $\Omega\in \mathbb{R}$ and $\theta\in \mathbb{T}$.
Finally, by virtue of   symmetry properties of the kernels $G_{k,j}$,it is then readily seen that preceding integral  is identically zero, thereby concluding that $(\Omega,0)$ is a zero for the contour dynamics, for any $\Omega\in \mathbb{R}$.
This completes the proof of the lemma.
\end{proof}

\begin{rem}
Let $(\omega_1,\omega_2)|_{t=0}=\big(\mathds{1}_{D_{1}},\mathds{1}_{D_{2}}\big)$ be a rotating patch to \eqref{EQ} with constant angular velocity $\Omega$. Further consider a real number $a>0$ and denote by $D_{k}^a=a D_{k}$, for $k\in\{1,2\}$. Then $(\omega_1,\omega_2)|_{t=0}=\big(\mathds{1}_{D_{1}^a},\mathds{1}_{D_{2}^a}\big)$ is also a rotating patch with angular velocity $\Omega$ to the system \eqref{EQ} with $\lambda$ replaced by $a\lambda$. Indeed, according to \eqref{boundary eq in}, \eqref{boundary eq in2} and \eqref{def streamL1}, one has that 
$$
\Omega\,\textnormal{Re}\Big\{ \overline{{z}_k(\theta)}\partial_{\theta}{z}_k(\theta)\Big\}=\partial_\theta\bigg( \sum_{j=1}^2\int_{D_j}G_{k,j}\big(z_k(\theta)-\xi\big)dA(\xi)\bigg).
$$
Thus, denoting $w_k=a z_k$, then multiplying the preceding equation by $a^2$  and implementing  the change of variables $\xi^\prime=a \xi$ yields that
$$
\Omega\,\textnormal{Re}\Big\{ \overline{{w}_k(\theta)}\partial_{\theta}{w}_k(\theta)\Big\}=\partial_\theta\bigg( \sum_{j=1}^2\int_{D_j^a}G_{k,j}\big(\tfrac1a w_k(\theta)-\tfrac1a\xi^\prime\big)dA(\xi^\prime)\bigg).
$$
Therefore, in view of \eqref{def Gkj}, we find that 
	\begin{align*}
	G_{k,j}\big(\tfrac1a w_k(\theta)-\tfrac1a\xi^\prime\big)&=-\frac{\delta^{2-j}}{2\pi(\delta+1)}\log(a)+\frac{\delta^{2-j}}{2\pi(\delta+1)}\log(|w_k(\theta)-\xi^\prime|)\\ &+(-1)^{k+j-1} \frac{\delta^{k-1}}{2\pi(\delta+1)} K_0\left(\tfrac{\lambda}{a} \sqrt{1+\delta} | w_k(\theta)-\xi^\prime|\right).  
	\end{align*}
Hence, 	we deduce that 
$$
\Omega\,\textnormal{Re}\Big\{ \overline{{w}_k(\theta)}\partial_{\theta}{w}_k(\theta)\Big\}=\partial_\theta\bigg( \sum_{j=1}^2\int_{D_j^a}\widetilde{G}_{k,j}\big( w_k(\theta)-\xi^\prime\big)dA(\xi^\prime)\bigg),
$$
where
$$
\widetilde{G}_{k,j}\big( w-\xi^\prime\big)\bydef \frac{\delta^{2-j}}{2\pi(\delta+1)}\log(|w-\xi^\prime|)+(-1)^{k+j-1} \frac{\delta^{k-1}}{2\pi(\delta+1)} K_0\left(\frac{\lambda}{a} \sqrt{1+\delta} | w-\xi^\prime|\right).
$$
\end{rem}

In the next step, we establish some  specific  symmetry properties of the functional $\mathcal{F}$, introduced in Lemma \ref{lem boundary eq}, above. In particular,  this is crucial and   will allow us to simplify the spectral analysis of the linearized operator associated with $\mathcal{F}$.

More importantly, it is to be emphasized that the property of $m$-fold symmetry  from the next lemma  will be useful to eliminate the instabilities and prevent their appearing at the linear level when we study spectral properties of the linearized operator associated with $\mathcal{F}$, later on.

\begin{lem}\label{lem symmetry}
Let  $\mathcal{F}=  (\mathcal{F}_1,\mathcal{F}_2)$ be given by \eqref{Fk:def}. Then, the following holds
\begin{enumerate}
\item The refection symmetry property: if $r=(r_1,r_2)$ are even, i.e, if   
	\begin{equation}\label{reflex r}
	  r (-\theta)=r (\theta),
	\end{equation}
  for any   $\theta\in \mathbb{T}$,
then $\mathcal{F}(\Omega,r )$ is odd, i.e.
$$
 \mathcal{F}(\Omega,r )(-\theta)=-  \mathcal{F}(\Omega,r )(\theta).  
$$
\item $m$-fold symmetry: if $r $ satisfy, for some $m\in \mathbb{Z}^*$ and all $\theta\in \mathbb{T}$, that
	\begin{equation*} 
  r \left(\theta+\tfrac{2\pi}{m}\right)=r (\theta),
	\end{equation*}
then, it holds that
$$
\mathcal{F}(\Omega,r )\left(\theta+\tfrac{2\pi}{m}\right)=\mathcal{F}(\Omega,r)(\theta),  $$
for all $\theta\in \mathbb{T}$.
\end{enumerate}

\end{lem}
\begin{proof}
 The justification of   reflection symmetry property of $\mathcal{F}$ relies upon the structure of the kernel $G_{k,j}$  defining $ \mathcal{F}_{k,j}$ in \eqref{def fkj}. Accordingly, by   performing the change of variables $\eta\mapsto-\eta$ in \eqref{def fkj}, it is readily seen, for any $k,j\in \{1,2\}$ and $\theta\in \mathbb{T}$, that	
 $$  \mathcal{F}_{k,j}(\Omega,r )(-\theta) =-  \mathcal{F}_{k,j}(\Omega,r )(\theta),$$
 whereby establishing the first claim in the lemma.  Likewise, the $m-$fold symmetry follows by making the change of variables $\eta\mapsto\eta+\tfrac{2\pi}{m}$ in \eqref{def fkj}, which also follows due to the symmetry property of the kernels $G_{k,j}$
 \begin{equation*}
 	G_{k,j}\big(e^{i\alpha}(x  -y) \big) = G_{k,j} (x  -y)  ,
 \end{equation*} 
 for any $x\neq  y\in \mathbb{C}$ and all $\alpha\in \mathbb{T}$. The proof of the lemma is now completed.
\end{proof}

 \subsection{Linearization around discs}\label{section:linearization}		 
 Here, we explore the structure of the linearized operator associated with the functional $\mathcal{F}$ defined in \eqref{def F}. We emphasize that the analysis we perform in this section will be rigorously justified by a detailed regularity study in Section \ref{sec regularity}, later on.

 \begin{lem}\label{lem lin op}
 The G\^ateaux derivative of $\mathcal{F}$ at $r=(r_1, r_2)$ in the direction $h=(h_1,h_2)$ is given  by
 \begin{equation}\label{defLr}
		d_r\mathcal{F}(\Omega,r)[h]= \Omega \begin{pmatrix}
			h'_1 & 0\\
			0 & h'_2
		\end{pmatrix} 
		+ \partial_\theta  \begin{pmatrix}
			  V_{1}(r) h_1- {L}_{1,1}(r)[h_1]& -{L}_{1,2}(r)[h_2] 
			  \vspace{5mm}\\
			- {L}_{2,1}(r)[h_1] &  V_{2}(r) h_2 -{L}_{2,2}(r)[h_2]
		\end{pmatrix}  ,
				\end{equation}
for any $\Omega\in \mathbb{R}$, where we set  
	\begin{align}\label{def Vk}
V_k(r)(\theta)&\bydef   \frac{1}{R_k(\theta)}\sum_{j=1}^{2}\int_{0}^{2\pi}G_{k,j}\big(R_k(\theta)e^{\ii \theta}-R_j(\eta) e^{\ii \eta}\big) \partial_\eta\Big( R_j(\eta)\sin(\eta-\theta) \Big)d\eta,
\\
L_{k,n}(r)[h_n](\theta)&\bydef   \int_{0}^{2\pi}G_{k,n}(R_{k}(\theta) e^{\ii \theta}-R_{n}(\eta) e^{\ii \eta})h_{n}(\eta) d\eta,\label{def Lkn}
	\end{align}  
	for any $\theta\in \mathbb{T}$.
 \end{lem}

 \begin{proof} Let   $\theta \in \mathbb{T} $ and $k\in \{1,2\}$ be fixed. Now, in view of Lemma \ref{lem boundary eq} above, notice that it is sufficient  to differentiate  the stream function. To that end, we first write, by virtue of  \eqref{def stream polar}, that   
 \begin{equation*}
 	d_{r_k} \psi_k(z)[h_k](\theta)=\int_{0}^{2\pi}G_{k,k}(z-R_k(\eta) e^{\ii \eta})h_k(\eta) d\eta
 \end{equation*}
 and
 \begin{equation*}
 	d_{r_{3-k}} \psi_k(z)[h_{3-k}](\theta)=\int_{0}^{2\pi}G_{k,3-k}(z-R_{3-k}(\eta) e^{\ii \eta})h_{3-k}(\eta) d\eta, 
 \end{equation*} 
for any  $z\in \mathbb{C}$. 
On the other hand, by differentiating \eqref{def zk} we infers that 
	$$d_{r_k}\overline{z}_k(\theta)[h_k](\theta)=\frac{h_k(\theta)}{R_k(\theta)}e^{-\ii\theta}.$$
Therefore, due to  \eqref{d z psi}, it follows that  
	\begin{align*}
2\textnormal{Re}\Big\{\partial_{\overline z} \psi_k\big({z}_k(\theta)\big) 	&d_{r_k} \overline{ {z}_k[h_k](\theta)}\Big\}\\ & \qquad=  -\frac{h_k(\theta)}{R_k(\theta)}\sum_{j=1}^2\int_{0}^{2\pi}G_{k,j}\big(R_k(\theta)e^{\ii \theta}-R_j(\eta) e^{\ii \eta}\big) \partial_\eta\, \textnormal{Im}\Big\{ R_j(\eta)e^{\ii (\eta-\theta)} \Big\}d\eta.
		\end{align*}
At last, combining   the foregoing identities yields that
\begin{equation*}
	d_{r_k}\big(\psi_k(z_k )\big)[h_k]=-V_k(r) h_k +L_{k,k}(r)[h_k] 
\end{equation*}
and
\begin{equation*}
d_{r_{3-k}}\big(\psi_k(z_k )\big)[h_{3-k}]=L_{k,3-k}(r)[h_{3-k}] ,
\end{equation*} 
thereby completing the proof of the lemma.
 \end{proof}

Before we move on to a more refined analysis of linearized operator $d_r \mathcal{F}$, allow us first to recall that  $\delta$ and $ \lambda$ are  considered here as fixed non-negative parameters defined by  the system of equations \eqref{EQ}, above. Accordingly, we recall the notations $$\mu\bydef   \lambda\sqrt{1+\delta}, \qquad \text{and} \qquad b \bydef \frac{b_2}{b_1}\in (0,1) , $$
where $b_1$ and $b_2$ refer to the radii of the initial patches. 
At last, we recall that the modified Bessel functions $I_n$ and $ K_n$, for $n\in \mathbb{N}$, are introduced  in Section \ref{section:Bessel} as well as their useful properties  which will come in handy later on in this section.

In the next lemma, we   compute the differential of $\mathcal{F}$ at the equilibrium $ (\Omega,0)$ and show that it acts as a Fourier multiplier.

  \begin{lem}\label{lemma linearized operator at 0}
Given $(h_1,h_2)$ with the Fourier series  expansions
$$ h_{k}(\theta)=\displaystyle\sum_{n\in \mathbb{N}^*}c_{n,k}\cos(n\theta) ,$$ 
for some sequence $(c_{n,k})_{n\in \mathbb{N}^*}\subset \mathbb{R}$, all $k\in \{1,2\}$ and $\theta\in \mathbb{T}$, it then holds that the linearized operator of $\mathcal{F}$  at $r=0$  writes, for any $\theta\in \mathbb{T}$, as
	\begin{equation*} 
	d_r\mathcal{F}(\Omega,0)\begin{pmatrix}
		{h}_{1}\vspace{0.1cm}\\
		{h}_{2}
	\end{pmatrix} (\theta)=-\sum_{n\in \mathbb{N}^*}n M_{n}(\Omega)\begin{pmatrix}
		c_{n,1}\vspace{0.1cm}\\
		c_{n,2}
	\end{pmatrix}\sin(n\theta), 
	\end{equation*}
	where we set
		\begin{equation*} 
	 M_{n}(\Omega)\bydef   \begin{pmatrix}
		\Omega +V_{b_1,b_2}+\tfrac{\tfrac{\delta}{2n}+I_{n}(b_1\mu)K_{n}(b_1\mu)}{\delta+1} & \tfrac{\tfrac{b^{n}}{2n}-I_{n}(\mu b_2)K_{n}(b_1\mu)}{\delta+1} \vspace{5mm}\\
		\tfrac{\delta}{\delta+1}\Big(\tfrac{ b^{n}}{2n}-  I_{n}(b_2\mu  )K_{n}(b_1\mu)\Big)& \Omega+W_{b_1,b_2}+\tfrac{\tfrac{1}{2n}+{\delta} I_{n}(b_2\mu)K_{n}(b_ 2\mu)}{\delta+1} 
	\end{pmatrix},
	\end{equation*}
	for all $n\in \mathbb{N}^*$,  and where $V_{b_1,b_2}$ and $W_{b_1,b_2}$ are defined in \eqref{V120}. 
	 	\end{lem}

\begin{proof}Let    $\theta\in \mathbb{T}$ and $k\in \{1,2\}$ be fixed in the proof and we proceed first by recasting the following identities: for any $n\in \mathbb{N}^*$ and $x,y, \lambda \in (0,\infty)$ with $x\leq y$, that   
		\begin{equation}\label{int log}
			\frac{1}{2\pi}  \int_{0}^{2\pi}\log\left(\big|1-xe^{\ii\theta}\big|\right)\cos(n\theta)d\theta =-\frac{x^{n}}{2n}
		\end{equation}
		and  
\begin{align}\label{int K0}
 \frac{1}{2\pi}\int_{0}^{2\pi}K_{0}\left(\lambda|x-y e^{\ii\theta}|\right)\cos(n\theta)d\theta=I_{n}(\lambda x)K_{n}(\lambda y).
\end{align}
	    The proof of the first identity can be found in \cite[Lemma A.3]{CCG16}, whereas the justification of the second one follows from the Beltrami's summation formula \cite[page 361]{W95}
	    \begin{equation*}
	    	 K_{0}\left(|x-y e^{\ii\theta}|\right) = \sum_{m\in \mathbb{Z}}  I_{m}( x)K_{m}( y) \cos(m\theta).
	    \end{equation*}  
 Now, observe that substituting the value $r=0$ in \eqref{def Vk}   gives
		\begin{align*} 
V_k(0)(\theta) =   \sum_{j=1}^2\frac{b_j}{b_k}\int_{0}^{2\pi}G_{k,j}\big(b_k e^{\ii \theta}-b_j e^{\ii \eta}\big) \cos(\eta-\theta) d\eta .
		\end{align*} 
		On the one hand, in view of  \eqref{def Gkj}, we write that 
			\begin{align*}
\int_{0}^{2\pi}G_{k,k}\big(b_k e^{\ii \theta}-b_k e^{\ii \eta}\big) \cos(\eta-\theta) d\eta&=\frac{\delta^{2-k}}{2\pi(\delta+1)} \int_{0}^{2\pi}\log\big(b_k \big| 1-e^{\ii (\eta-\theta)}\big|\big) \cos(\eta-\theta) d\eta\\ &\quad -\frac{\delta^{k-1}}{2\pi(\delta+1)}\int_{0}^{2\pi}K_0\big(\mu b_k  \big| 1-e^{\ii (\eta-\theta)}\big|\big) \cos(\eta-\theta) d\eta.
		\end{align*}		
		Therefore,   changing  the variable $\eta\mapsto\eta+\theta$ and employing the identities  \eqref{int K0} and \eqref{int log} entails that
		\begin{align*} 
\int_{0}^{2\pi}G_{k,k}\big(b_k e^{\ii \theta}-b_k e^{\ii \eta}\big) \cos(\eta-\theta) d\eta&=-\frac{\delta^{2-k}}{2(\delta+1)}  -\frac{\delta^{k-1}}{\delta+1}I_{1}(b_k\mu)K_{1}(b_k\mu).
		\end{align*} 
		On the other hand, by virtue of \eqref{def Gkj}, writing  
		\begin{align*}
\int_{0}^{2\pi}G_{k,3-k}\big(b_k e^{\ii \theta}-b_{3-k} e^{\ii \eta}\big) &\cos(\eta-\theta) d\eta
\\
&=\frac{\delta^{k-1}}{2\pi(\delta+1)}\int_{0}^{2\pi}\log\big(b_1 \big| 1-b e^{\ii (\eta-\theta)}\big|\big)\cos(\eta-\theta) d\eta\\ &\quad+\frac{\delta^{k-1}}{2\pi(\delta+1)}\int_{0}^{2\pi}K_0\big(\mu  \big| b_k-b_{3-k} e^{\ii (\eta-\theta)}\big|\big) \cos(\eta-\theta) d\eta, 
		\end{align*}
 	it then follows, by the same arguments as before, that 
 	\begin{align*} 
\int_{0}^{2\pi}G_{k,3-k}\big(b_k e^{\ii \theta}-b_{3-k} e^{\ii \eta}\big) \cos(\eta-\theta) d\eta&=-\frac{\delta^{k-1} b}{2(\delta+1)}+\frac{\delta^{k-1}}{\delta+1}  I_1(b_2\mu)K_1(b_1\mu) 		.
\end{align*} 
 Next, according to \eqref{def Lkn}  and  \eqref{def Gkj}, one has that
	\begin{align*}
	L_{k,j}(0)[h_j](\theta)&\bydef   \sum_{n\geq 1}c_{n,j}\int_{0}^{2\pi}G_{k,j}(b_{k} e^{\ii \theta}-b_j e^{\ii \eta})\cos(n\eta) d\eta\\ &=\sum_{n\geq 1}\tfrac{c_{n,j}}{2\pi(\delta+1)}\bigg[\delta^{2-j}\int_{0}^{2\pi}\log\big(\big|b_{k} e^{\ii \theta}-b_j e^{\ii \eta}\big|\big)\cos(n\eta) d\eta\\ &\qquad - (-1)^{k+j} \delta^{k-1} \int_{0}^{2\pi}K_0\big(\mu\big|b_{k} e^{\ii \theta}-b_j e^{\ii \eta}\big|\big)\cos(n\eta) d\eta\bigg],
	\end{align*}
	  for any $j\in \{1,2\}$. Thus, by a change of variables, we find that 
		\begin{align*}
	L_{k,j}(0)[h_j](\theta) &=\sum_{n\geq 1}\frac{c_{n,j}}{2\pi(\delta+1)}\bigg(\delta^{2-j}\int_{0}^{2\pi}\log\big(\big|b_{k} -b_j e^{\ii \eta}\big|\big)\cos(n\eta+n\theta) d\eta\\ &\qquad - (-1)^{k+j} \delta^{k-1} \int_{0}^{2\pi}K_0\big(\mu\big|b_{k} -b_j e^{\ii \eta}\big|\big)\cos(n\eta+n\theta) d\eta\bigg)\\
	&=\sum_{n\geq 1}\frac{c_{n,j}}{2\pi(\delta+1)}\bigg(\delta^{2-j}\cos(n\theta)\int_{0}^{2\pi}\log\big(\big|b_{k} -b_j e^{\ii \eta}\big|\big)\cos(n\eta) d\eta\\ 
	&\qquad- \delta^{2-j}\sin(n\theta)\int_{0}^{2\pi}\log\big(\big|b_{k} -b_j e^{\ii \eta}\big|\big)\sin(n\eta) d\eta\\ &\qquad - (-1)^{k+j} \delta^{k-1}\cos(n\theta) \int_{0}^{2\pi}K_0\big(\mu\big|b_{k} -b_j e^{\ii \eta}\big|\big)\cos(n\eta) d\eta
	\\ &\qquad - (-1)^{k+j} \delta^{k-1} \sin(n\theta)\int_{0}^{2\pi}K_0\big(\mu\big|b_{k} -b_j e^{\ii \eta}\big|\big)\sin(n\eta) d\eta\bigg).
	\end{align*}
Therefore, noting, by symmetry,  that 	
\begin{align*}
\int_{0}^{2\pi}\log\big(\big|b_{k} -b_j e^{\ii \eta}\big|\big)\sin(n\eta) d\eta&=0 
	\end{align*}
	and
	\begin{align*} 
\int_{0}^{2\pi}K_0\big(\mu\big|b_{k} -b_j e^{\ii \eta}\big|\big)\sin(n\eta) d\eta&=0,
	\end{align*}
we conclude, by utilizing   \eqref{int K0} and \eqref{int log}, again, that
		\begin{align*}
L_{k,k}(0)[h_k](\theta)&=-\sum_{n\geq 1}\frac{c_{n,k}}{\delta+1}\left(\frac{\delta^{2-k}}{2n} +\delta^{k-1} I_n(b_k\mu)K_n(b_k\mu)\right)\cos(n\theta),
		\\
L_{k,3-k}(r)[h_{3-k}](\theta)&=\sum_{n\geq 1}\frac{c_{n,3-k}\delta^{k-1} }{\delta+1}\left(-\frac{b^n}{2n} +I_n(b_2\mu)K_n(b_1\mu)\right)\cos(n\theta).		
	\end{align*}
At last, gathering the foregoing identities and plugging them in \eqref{defLr} at the equilibrium  $r=0$ completes the proof of the lemma.	
\end{proof}

 \subsection{Regularity properties}\label{sec regularity}
 Here, we justify the   regularity properties   of the nonlinear functional $\mathcal{F}$ introduced
in \eqref{def F}. In particular, by virtue of the results laid out in this section, all the formal computations in the Section \ref{section:linearization} above will be fully justified.
 
  Let us first set  a few notations to be used afterwards. For $\alpha\in(0,1)$, and $m\in\mathbb{N}^+$, consider the $m$-fold Banach  spaces
  \begin{equation*}
  	 X_m^\alpha \bydef    \Big\{h\in C^{1+\alpha}(\mathbb{T}):\, h(\theta)=\sum_{n\geq 1}c_n\cos(nm\theta),\, c_n\in \mathbb{R},
  \, \theta\in\mathbb{T}\Big\} 
  \end{equation*}
  and
  \begin{equation*}
  	Y_m^\alpha \bydef    \Big\{h\in C^{\alpha}(\mathbb{T}):\, h(\theta)=\sum_{n\geq 1}c_n\sin(nm\theta),\, c_n\in \mathbb{R},
  \, \theta\in\mathbb{T}\Big\}
  \end{equation*} 
  equipped  with their usual norms. Accordingly, we define the product spaces 
$$
\mathcal{X}_m^\alpha\bydef    X_m^\alpha\times X_m^\alpha,\qquad \mathcal{Y}_m^\alpha\bydef   Y_m^\alpha\times Y_m^\alpha,
$$
and, for all $\epsilon\in (0,1)$, we denote by $\mathcal{B} _{m,\epsilon}^\alpha$ the open ball in $\mathcal{X}_m^\alpha$ centered at the origin and with radius $\epsilon$, i.e.,
\begin{align*}
  \mathcal{B} _{m,\epsilon}^\alpha \bydef \big\{r\in \mathcal{X} _m: \lVert r\rVert_{\mathcal{X}_m} < \epsilon \big\}.
\end{align*}

All regularity properties of the functional $\mathcal{F}$ that we need  are now established in the next proposition.

\begin{prop}\label{proposition regularity of the functional}
Let $\lambda>0,$ $b\in(0,1)$, $\alpha\in(0,1)$ and $m\in\mathbb{N}^*.$ Then, there exists $\epsilon >0$ such that  the functional  $\mathcal{F}$ introduced in  \eqref{def F}   is well defined as a mapping
$$\mathcal{F}:\mathbb{R}\times \mathcal{B} _{m,\epsilon}^\alpha\rightarrow \mathcal{Y}_m^\alpha,$$
  and  is of class $C^{1}.$  
Moreover, the partial derivative $$\partial_{\Omega,r}^2\mathcal{F}:\mathbb{R}\times  B_{m,\epsilon}^\alpha\rightarrow\mathcal{L}(\mathcal{X}_m^\alpha,\mathcal{Y}_m^\alpha)$$ exists and is continuous. 
\end{prop}
\begin{proof}

Throughout the proof, $\alpha\in (0,1)$ and   $r \in B_{m,\varepsilon}^\alpha $ are fixed parameters.  In view of \eqref{def F}, notice that   it suffices to establish  the regularity properties for $\mathcal{F}_{k,j}$, for all $j,k\in \{1,2 \}$, where $\mathcal{F}_{k,j}$ are introduced in \eqref{def fkj} and can be recast here as 
 $$\mathcal{F}_{k,j} (r(\theta))  = \ii \partial_\theta\big(R_k(\theta)e^{\ii \theta} \big) \cdot    \int_{0}^{2\pi}G_{k,j}\big(R_k(\theta)e^{\ii \theta}-R_j(\eta) e^{\ii \eta}\big)   \partial_\eta\big(R_j(\eta)e^{\ii \eta} \big) d\eta,$$
 where we recall that
  $$ R_k(\theta) = \sqrt{b_k^{2}+2r_k(\theta) } ,$$
 for all $k\in \{1,2\}$. 
 Thus, it is readily seen that 
 $$ \theta \mapsto  \partial_\theta\big(R_k(\theta)e^{\ii \theta} \big) \in C^\alpha(\mathbb{T}),$$
 as soon as the deformation $\theta \mapsto r(\theta)$ belongs to $C^{1+\alpha}(\mathbb{T})$, which holds by assumption. 
Next, observe, for $ r \in \mathcal{X}_m^{\alpha}$, that 
\begin{equation}\label{bound:Rk}
 b_k \leq  |R_k(\theta)| \leq \sqrt{ b_k^2 + 2 \epsilon}, \quad \text{for all} \quad \theta \in \mathbb{T}.
\end{equation}
Therefore, in view of the identities \eqref{def Gkj} and \eqref{expan K0}, altogether with the fact that 
$$ \big| \log |x| \big|\lesssim |x|^{-\alpha}, $$
for  all $x\in \mathbb{R}^2\setminus \{0\}$ and any $\alpha\in (0,1),$
it is then  readily seen that 
$$ \bigg|G_{ k,j} \big(R_{k}(\theta) e^{\ii \theta}- R_{n}(\eta) e^{\ii \eta}\big)  \bigg|\lesssim \big|R_{k}(\theta) e^{\ii \theta}- R_{n}(\eta) e^{\ii \eta}\big| ^{-\alpha} ,$$
 for any $j,k\in \{ 1,2\}$.
Accordingly, one deduces for $\epsilon $ small enough (see for \cite[inequality (61)]{HXX}) that 
 \begin{equation}\label{G:ES}
 \Big| G_{k,n}\big(R_{k}(\theta) e^{\ii \theta}- R_{n}(\eta) e^{\ii \eta}\big) \Big| \lesssim \Big| \sin \left( \tfrac{\theta-\eta}{2} \right)\Big| ^ {-\alpha},  
\end{equation} 
for all  $\theta\neq \eta \in \mathbb{T}$. 
	Moreover, noticing that 
$$ \Big| \nabla \left(  G_{ k,j}  (|x|) \right)   \Big| \lesssim   | k_{+}(x)| + | k_{-}(x)|,$$
where $k_{\pm}$ are introduced in \eqref{k:pm:def}, 
and employing  \eqref{A1} yields that 
$$ \Big| \nabla \left(  G_{ k,j}  (|x|) \right)   \Big| \lesssim  |x|^{-1}.$$
Thus, in view of \eqref{bound:Rk}, we deduce that 
\begin{equation}\label{DG:ES}
\Big| \partial_\theta\left( G_{k,j}(R_{k}(\theta) e^{\ii \theta}- R_{j}(\eta) e^{\ii \eta}) \right) \Big| \lesssim \Big| \sin \left( \tfrac{\theta-\eta}{2} \right)\Big| ^ {-(1+\alpha)},  
\end{equation}
for all $  \theta\neq \eta \in \mathbb{T} $. 
All in all, with \eqref{G:ES} and \eqref{DG:ES} in hand,    Lemma \ref{lemma:regularity-operator} allows us to deduce that  $$  \theta \mapsto \mathcal{F}_{k,j} (r(\theta)) \in  C^\alpha(\mathbb{T}),$$ 
for any $j,k\in \{1,2\}$.  
In addition to that, the reflection symmetry of $\mathcal{F}_{k}$ is already established in \eqref{reflex r}, whereby we arrive at the conclusion that the mapping 
$$ \mathcal{F}: \mathbb{R}\times B_{m,\varepsilon}^\alpha \rightarrow \mathcal{Y}_m^\alpha $$
is well defined.  
Now, we show that this mapping is $C^1$ with respect to the variable $r$, as its $C^1$ regularity wit respect to variable $\Omega$ clearly holds true as it follows from the observation that 
$$ \partial_\Omega \mathcal{F}(\Omega,r) = r'(\theta).$$
 To that end, allow us first to  recall the expression \eqref{defLr} here for convenience 
  \begin{equation*} 
		d_r\mathcal{F}(\Omega,r)[h]= \Omega \begin{pmatrix}
			h'_1 & 0\\
			0 & h'_2
		\end{pmatrix} 
		+ \partial_\theta  \begin{pmatrix}
			  V_{1}(r) h_1- {L}_{1,1}(r)[h_1]& -{L}_{1,2}(r)[h_2] 
			  \vspace{5mm}\\
			- {L}_{2,1}(r)[h_1] &  V_{2}(r) h_2 -{L}_{2,2}(r)[h_2]
		\end{pmatrix}  ,
				\end{equation*} 
				which holds for any $h=(h_1,h_2)\in \mathcal{X}_m^\alpha$,
				where $V_k$ and $L_{k,n}$ are respectively given by \eqref{def Vk} and \eqref{def Lkn} which can also be   recast   as   
				\begin{equation*}
					V_k(r)(\theta)=  \frac{e^{i\theta}}{R_k(\theta)} \cdot \sum_{j=1}^{2}\int_{0}^{2\pi}G_{k,j}\big(R_k(\theta)e^{\ii \theta}-R_j(\eta) e^{\ii \eta}\big) \partial_\eta\Big( R_j(\eta)e^{i\eta} \Big)d\eta
				\end{equation*}
				and 
				\begin{equation*} 
L_{k,n}(r)[h_n](\theta)=   \int_{0}^{2\pi}G_{k,n}(R_{k}(\theta) e^{\ii \theta}- R_{n}(\eta) e^{\ii \eta})h_{n}(\eta) d\eta .
				\end{equation*} 
				
				Now, we shall prove that $d_r\mathcal{F}(\Omega,r)[\cdot] $ is a well defined linear mapping from $\mathcal{X}_m^\alpha$ into $\mathcal{Y}_m^\alpha$ and we emphasize that we can restrict our focus on   the    regularity property, for the symmetry follows  by similar arguments to the proof of  Lemma \ref{lem symmetry}.  Moreover, observe that the $C^\alpha$ regularity of  $d_r\mathcal{F}(\Omega,r)[h]$, for any $h\in C^{\alpha+1}\times C^{\alpha +1}$, directly follows from the $C^{\alpha+1}$ regularity of $V_k(r)[h]$ and $L_{k,n}(r)[h]$, which we will now prove in details. 
Thus, we now claim that  $$ \theta \mapsto \partial_\theta \int_{0}^{2\pi}G_{k,j}\big(R_k(\theta)e^{\ii \theta}-R_j(\eta) e^{\ii \eta}\big) \partial_\eta\Big( R_j(\eta)e^{i\eta} \Big)d\eta $$
and 
$$  \theta \mapsto  \partial_\theta \int_{0}^{2\pi}G_{k,j}(R_{k}(\theta) e^{\ii \theta}- R_{j}(\eta) e^{\ii \eta})h_{j}(\eta) d\eta$$
belong to $C^\alpha (\mathbb{T})$ and we split the proof of that into two parts:

\subsubsection*{Regularity of the anti-diagonal} This is the case where the kernels $G_{j,k}$ are regular due to a crucial cancellation that only appears  in the elements of the anti-diagonal  of the gradient-matrix. This   cancellation holds to be a consequence of the specific  combination of the kernels associated with $\Delta^{-1}$ and $(\Delta - \lambda^2 (1+\delta) \id )^{-1} $ through the coupling in     \eqref{EQ}.  

In order  to observe this phenomenon,  we first emphasize that  the   expression of $G_{j,k}$ in \eqref{def Gkj}, for $j\neq k,$ can be  recast as
$$G_{1,2} (z) = \frac{1}{2\pi(1+\delta)}  \Big( \log(|z|) +  K_0(\mu |z|) \Big)$$
and
$$G_{2,1} (z) = \frac{\delta}{2\pi(1+\delta)}  \Big( \log(|z|) +  K_0(\mu |z|) \Big).$$ 
Therefore, we expand $K_0$ by utilizing   \eqref{In:def} to find, after performing minor simplifications, that 
\begin{equation}\label{K0:expansion}
	K_0(\mu |z|)  = -  \log (|z|)    -  \left( \underbrace{\log \left( \frac{\mu}{2}\right) I_0(\mu |z|) + \log \left( \frac{|z|}{2}\right) \sum_{ m=1}^\infty \frac{\left( \mu \frac{ | z|}{2}\right)^{2m}}{m! \Gamma ( m + 1)}}_{\bydef Q(|z|)} \right) .
\end{equation}
Thus, we deduce that  
$$G_{1,2} (z) = \frac{1}{2\pi(1+\delta)}  Q(|z|)   $$
$$G_{2,1} (z) = \frac{\delta}{2\pi(1+\delta)}   Q(|z|) .$$ 
The crucial observation here is that $Q$ is more regular than the previous kernels, $\log  $ and $K_0$, themselves.  
More precisely, we emphasize that one can show by a direct computations that 
$$ Q(|\cdot|) \in C^{1+\alpha}_{\loc}(\mathbb{R}^+).   $$
Therefore, it is readily seen that 
$$ \theta \mapsto \partial_\theta \int_{0}^{2\pi}G_{k,j}\big(R_k(\theta)e^{\ii \theta}-R_j(\eta) e^{\ii \eta}\big) \partial_\eta\Big( R_j(\eta)e^{i\eta} \Big)d\eta $$
and 
$$  \theta \mapsto  \partial_\theta \int_{0}^{2\pi}G_{k,j}(R_{k}(\theta) e^{\ii \theta}- R_{j}(\eta) e^{\ii \eta})h_{j}(\eta) d\eta$$
belong to $C^\alpha (\mathbb{T}),$ for any $i\neq j \in \{ 1,2\}$, as soon as $r\in C^{1+\alpha}(\mathbb{T})$, which holds by assumption.
\subsubsection*{Regularity of the diagonal}

Owing again to  \eqref{def Gkj} and \eqref{K0:expansion}, we write that 
$$G_{1,1} (z) = \frac{1}{2\pi(1+\delta)}  \Big( (\delta+1)\log(|z|) -  Q(|z|) \Big)$$
and
$$G_{2,2} (z) = \frac{1}{2\pi(1+\delta)}  \Big( (\delta+1)\log(|z|) -  \delta Q(|z|) \Big),$$ 
where the function $Q$ is defined in \eqref{K0:expansion} and belongs to $C^{1+\alpha}(\mathbb{T})$ as it is emphasized above. Accordingly, in order to show the regularity of the diagonal elements, i.e. the fact that 
$$ \theta \mapsto \partial_\theta \int_{0}^{2\pi}G_{k,k}\big(R_k(\theta)e^{\ii \theta}-R_k(\eta) e^{\ii \eta}\big) \partial_\eta\Big( R_k(\eta)e^{i\eta} \Big)d\eta $$
and 
$$  \theta \mapsto  \partial_\theta \int_{0}^{2\pi}G_{k,k}(R_{k}(\theta) e^{\ii \theta}- R_{k}(\eta) e^{\ii \eta})h_{k}(\eta) d\eta$$
belong to $C^\alpha (\mathbb{T}),$ for any $k \in \{ 1,2\}$, it only remains to prove that 
$$ \theta \mapsto \partial_\theta \int_{0}^{2\pi} \log \Big|R_k(\theta)e^{\ii \theta}-R_k(\eta) e^{\ii \eta}\Big| \partial_\eta\Big( R_k(\eta)e^{i\eta} \Big)d\eta $$
and 
$$  \theta \mapsto  \partial_\theta \int_{0}^{2\pi}\log \Big|R_{k}(\theta) e^{\ii \theta}- R_{k}(\eta) e^{\ii \eta}  \Big| h_{k}(\eta) d\eta$$
belong to $C^\alpha (\mathbb{T})$. 
Note that this is equivalent to showing that the functions 
$$ \theta \mapsto \mathcal{U}(\theta)\bydef  \int_{0}^{2\pi} \frac{R_k(\theta)e^{\ii \theta}-R_k(\eta) e^{\ii \eta}}{\Big| R_k(\theta)e^{\ii \theta}-R_k(\eta) e^{\ii \eta}\Big|^2}   \cdot\partial_\theta \Big( R_k(\theta)e^{\ii \theta}\Big) \partial_\eta\Big( R_k(\eta)e^{i\eta} \Big)d\eta $$
and 
$$  \theta \mapsto \mathcal{W}(\theta)\bydef    \int_{0}^{2\pi}\frac{R_k(\theta)e^{\ii \theta}-R_k(\eta) e^{\ii \eta}}{\Big| R_k(\theta)e^{\ii \theta}-R_k(\eta) e^{\ii \eta}\Big|^2}   \cdot\partial_\theta \Big( R_k(\theta)e^{\ii \theta}\Big)  h_{k}(\eta) d\eta$$

 enjoy the $C^\alpha$ regularity.  
 The achievement of the preceding claims relies on the observation that
\begin{equation}\label{identity:zero}
	\int_{0}^{2\pi}\frac{R_k(\theta)e^{\ii \theta}-R_k(\eta) e^{\ii \eta}}{\Big| R_k(\theta)e^{\ii \theta}-R_k(\eta) e^{\ii \eta}\Big|^2}   \cdot\partial_\eta \Big( R_k(\eta)e^{\ii \eta}\Big)   d\eta = 0,
\end{equation}
for any $k \in \{1,2\}$, which, in particular, allows us to write, for all $\theta \in \mathbb{T}$, that
\begin{equation}\label{U:expansion}
	\begin{aligned}
		\mathcal{U}(\theta) 
		&=  \partial_\theta \Big( R_k(\theta)e^{\ii \theta}\Big)  \cdot \int_{0}^{2\pi} \frac{R_k(\theta)e^{\ii \theta}-R_k(\eta) e^{\ii \eta}}{\Big| R_k(\theta)e^{\ii \theta}-R_k(\eta) e^{\ii \eta}\Big|^2}   \Bigg(\partial_\eta\Big( R_k(\eta)e^{i\eta} \Big) - \partial_\theta \Big( R_k(\theta)e^{\ii \theta}\Big)  \Bigg)d\eta  
		\\
		 & \quad + \partial_\theta \Big( R_k(\theta)e^{\ii \theta}\Big)  \int_{0}^{2\pi} \frac{R_k(\theta)e^{\ii \theta}-R_k(\eta) e^{\ii \eta}}{\Big| R_k(\theta)e^{\ii \theta}-R_k(\eta) e^{\ii \eta}\Big|^2} \cdot  \Bigg( \partial_\theta \Big( R_k(\theta)e^{\ii \theta}\Big)  -\partial_\eta\Big( R_k(\eta)e^{i\eta} \Big) \Bigg)d\eta  
	\end{aligned}
	\end{equation}
	and that 
	\begin{equation*}
		\begin{aligned}
		\mathcal{W}(\theta) 
			&= \partial_\theta \Big( R_k(\theta)e^{\ii \theta}\Big) \cdot \int_{0}^{2\pi}\frac{R_k(\theta)e^{\ii \theta}-R_j(\eta) e^{\ii \eta}}{\Big| R_k(\theta)e^{\ii \theta}-R_k(\eta) e^{\ii \eta}\Big|^2}   \Big(    h_{k}(\eta) - h_k(\theta) \Big) d\eta
			\\
			& \quad +  h_k(\theta) \int_{0}^{2\pi} \frac{R_k(\theta)e^{\ii \theta}-R_k(\eta) e^{\ii \eta}}{\Big| R_k(\theta)e^{\ii \theta}-R_k(\eta) e^{\ii \eta}\Big|^2} \cdot  \Bigg( \partial_\theta \Big( R_k(\theta)e^{\ii \theta}\Big)  -\partial_\eta\Big( R_k(\eta)e^{i\eta} \Big) \Bigg)d\eta .
		\end{aligned}
	\end{equation*} 
	On the other hand, due to the $C^{1+\alpha}$ regularity of $h_j$ and $r$, one has that 
	\begin{equation*}
		\Bigg| \partial_\theta \Big( R_k(\theta)e^{\ii \theta}\Big)  -\partial_\eta\Big( R_k(\eta)e^{i\eta} \Big) \Bigg| \lesssim \big|\theta - \eta \big|^\alpha
	\end{equation*}
	and that 
	\begin{equation*}
		\Big| h_k(\eta) - h_k(\theta) \Big|\lesssim \big|\theta - \eta \big|^\alpha,
	\end{equation*}
	for any $\theta,\eta \in \mathbb{T}$.  Therefore, combining the preceding bounds with  
	\begin{equation}\label{Bound:below}
		\left| \frac{R_k(\theta)e^{\ii \theta}-R_k(\eta) e^{\ii \eta}}{\Big| R_k(\theta)e^{\ii \theta}-R_k(\eta) e^{\ii \eta}\Big|^2}\right| \lesssim \frac{1}{\Big| R_k(\theta)e^{\ii \theta}-R_k(\eta) e^{\ii \eta}\Big| }  \lesssim  \Big| \sin \left( \tfrac{\theta-\eta}{2} \right)\Big| ^ {-1}
	\end{equation}
	and 
	\begin{equation*}
		\left|  \partial_ \theta \left(\frac{R_k(\theta)e^{\ii \theta}-R_k(\eta) e^{\ii \eta}}{\Big| R_k(\theta)e^{\ii \theta}-R_k(\eta) e^{\ii \eta}\Big|^2}\right)\right| \lesssim \Big| \sin \left( \tfrac{\theta-\eta}{2} \right)\Big|^{-2},
	\end{equation*}
	which can be proved by means of a straitforward computations,  
	 the desired regularity of $\mathcal{U}$ and $\mathcal{W}$ follows then by a direct application of Lemma \ref{lemma:regularity-operator2}.  
	This completes the proof of the fact that the operator  $ d_r \mathcal{F}[\cdot]$ is well defined  from $\mathbb{R}\times B_{m,\epsilon}^\alpha$ to $C^{\alpha}(\mathbb{T})$.
	
	As for its  continuity with respect to the $r$ variable, we emphasize that we can focus on  establishing  that for the functions  
	$$ \theta \mapsto \partial_\theta \int_{0}^{2\pi}G_{k,k}\big(R_k(\theta)e^{\ii \theta}-R_k(\eta) e^{\ii \eta}\big) \partial_\eta\Big( R_k(\eta)e^{i\eta} \Big)d\eta $$
and 
$$  \theta \mapsto  \partial_\theta \int_{0}^{2\pi}G_{k,k}(R_{k}(\theta) e^{\ii \theta}- R_{k}(\eta) e^{\ii \eta})h_{k}(\eta) d\eta,$$
as the treatment of the remaining terms in the expression of $ d_r\mathcal{F}[h]$ is straitforward. 
Again, this reduces to proving the continuity, with respect to $r$, of   the functions $\mathcal{U}$ and $\mathcal{W}$ defined above.  

For simplicity, we only outline here the proof of that for $\mathcal{U}$, for the same arguments apply for $\mathcal{W}$, as well.    
To that end, we  consider two deformations $r=(r_1,r_2)$ and $\widetilde{r}= (\widetilde{r}_1, \widetilde{r}_2) $ and we  are going to use the notation  $\widetilde{f}$, for any given functional $f$, to precise its dependence   on $\widetilde{r}$ instead of $r$. Moreover, let us introduce the functions
\begin{equation*}
	M(\theta) \bydef R_k(\theta)e^{\ii \theta}, \qquad \mathcal{J}_{k}(\theta )\bydef M'(\theta), 
\end{equation*}
\begin{equation*}
	\mathcal{H}_{k }(\theta,\eta)\bydef \frac{M(\theta)-M(\eta)}{\Big| M(\theta)-M(\eta)\Big|^2},
\end{equation*} 
and we recall that 
$$R_k= \sqrt{b_k^2 + 2 r_k}.$$ 
In view of these notations,  we  recast \eqref{identity:zero}  as 
\begin{equation*}
	\int_0^{2\pi} \mathcal{H}_{k }(\theta,\eta) \cdot \mathcal{J}_{k}(\eta) d\eta =0,
\end{equation*}
for all $\theta\in \mathbb{T}$.
Therefore, we write, in view of \eqref{U:expansion}, that 
\begin{equation*} 
	\begin{aligned}
		\mathcal{U}(\theta) - \mathcal{\widetilde{U}}(\theta) 
		&= \mathcal{J}_{k}(\theta ) \cdot \int_{0}^{2\pi}  \mathcal{H}_{k }(\theta,\eta)   \Big(\mathcal{J}_{k}(\eta ) - \mathcal{J}_{k}(\theta )  \Big)d\eta
		\\  
		 & \quad +\mathcal{J}_{k}(\theta )  \int_{0}^{2\pi}  \mathcal{H}_{k }(\theta,\eta) \cdot  \Big(\mathcal{J}_{k}(\theta )  -\mathcal{J}_{k}(\eta ) \Big)d\eta   
		 \\
		 &\quad - \mathcal{\widetilde{J}}_{k}(\theta ) \cdot \int_{0}^{2\pi}  \mathcal{\widetilde{H}}_{k }(\theta,\eta)   \Big(\mathcal{\widetilde{J}}_{k}(\eta ) - \mathcal{\widetilde{J}}_{k}(\theta )  \Big)d\eta  
		 \\
	    &\quad - \mathcal{\widetilde{J}}_{k}(\theta )  \int_{0}^{2\pi}  \mathcal{\widetilde{H}}_{k }(\theta,\eta) \cdot  \Big(\mathcal{\widetilde{J}}_{k}(\theta )  -\mathcal{\widetilde{J}}_{k}(\eta ) \Big)d\eta 
	  \\
	   &= \big( \mathcal{J}_{k}(\theta ) - \mathcal{\widetilde{J}}_{k}(\theta )  \big) \cdot \int_{0}^{2\pi}  \mathcal{H}_{k }(\theta,\eta)   \Big(\mathcal{J}_{k}(\eta ) - \mathcal{J}_{k}(\theta )  \Big)d\eta
	   \\
	     &\quad +  \mathcal{\widetilde{J}}_{k}(\theta ) \cdot \int_{0}^{2\pi} \big(  \mathcal{H}_{k }(\theta,\eta) -  \mathcal{\widetilde{H}}_{k }(\theta,\eta)  \big)    \Big(\mathcal{J}_{k}(\eta ) - \mathcal{J}_{k}(\theta )  \Big)d\eta  
	     \\
	     &\quad +  \mathcal{\widetilde{J}}_{k}(\theta ) \cdot \int_{0}^{2\pi}   \mathcal{\widetilde{H}}_{k }(\theta,\eta)      \Big(\big(\mathcal{J}_{k}-\mathcal{\widetilde{J}}_{k}\big) (\eta ) - \big(\mathcal{J}_{k}-\mathcal{\widetilde{J}}_{k}\big)(\theta )  \Big)d\eta  
	     \\ 
	     & \bydef \sum_{j=1}^3 \mathcal{A}_j (\theta).
	\end{aligned}
	\end{equation*}
Thus, by virtue of the same argument laid out in the regularity study of   $\mathcal{U}$ in the previous step of the proof, we observe that a direct application of Lemma  \ref{lemma:regularity-operator2} leads to the bound 
\begin{equation*}
	\norm { \mathcal{A}_1}_{C^{\alpha}} + \norm { \mathcal{A}_3}_{C^{\alpha}} \lesssim \|\mathcal{J}_k-\mathcal{\widetilde{J}}_{k}\|_{C^{\alpha}} \lesssim \norm{ r_k - \widetilde{r}_k}_{C^{1+\alpha}}.
\end{equation*} 
As for the estimate of $A_2$, in order to apply Lemma \ref{lemma:regularity-operator2},  we first need to show that
\begin{equation}\label{H:BOUND:1}
	\big| \big( \mathcal{H}_{k }- \mathcal{\widetilde{H}}_{k } \big) (\theta,\eta)\big| \lesssim \norm { r-\widetilde{r}}_{C^{1+\alpha}} \Big| \sin \left( \tfrac{\theta-\eta}{2} \right)\Big|^{-1}
\end{equation}
and 
\begin{equation}\label{H:BOUND:2}
	\big|\partial_\theta \big( \mathcal{H}_{k }- \mathcal{\widetilde{H}}_{k } \big) (\theta,\eta)\big| \lesssim \norm { r-\widetilde{r}}_{C^{1+\alpha}} \Big| \sin \left( \tfrac{\theta-\eta}{2} \right)\Big|^{-2},
\end{equation}
for all $\theta \neq \eta \in \mathbb{T}$.  To that end, we only need to write, by a direct computation, that 
\begin{equation*}
	\begin{aligned}
	\big| 	 \mathcal{H}_{k }(\theta,\eta)  & -  \mathcal{\widetilde{H}}_{k }(\theta,\eta) \big|   
	\\
	 & \lesssim  \big|  \big(M-\widetilde{M}\big)(\theta)- \big(M-\widetilde{M}\big)(\eta)  
\big| \left(  \frac{1}{ \Big| {M}(\theta) - {M}(\eta) \Big| }+ \frac{1}{ \Big| \widetilde{M}(\theta) - \widetilde{M}(\eta) \Big| }\right) ^2.
			\end{aligned}
\end{equation*} 
Therefore, by using   the Lipschitz property of the functions $M$ and $\widetilde M$ combined with \eqref{Bound:below}, one deduces   that 
\begin{equation*}
	\begin{aligned}
	\big|\mathcal{H}_{k }(\theta,\eta)    -  \mathcal{\widetilde{H}}_{k }(\theta,\eta) \big|   
	& \lesssim \|M-\widetilde{M}\|_{C^1} \frac{|\theta-\eta|}{ \Big| \sin \left( \frac{\theta-\eta}{2} \right)\Big|^2} 
	\\
	& \lesssim  \norm { r-\widetilde{r}}_{C^{1+\alpha}} \Big| \sin \left( \frac{\theta-\eta}{2} \right)\Big|^{-1},
			\end{aligned}
\end{equation*}
whereby showing \eqref{H:BOUND:1}. The justification of \eqref{H:BOUND:2} can be done along the same lines, whence we skip its proof here. This completes the proof of the proposition.
   \end{proof}

  \subsection{Spectral analysis of the linearized operator} This section is devoted, first, to perform a refine analysis of the eigenvalues associated with the matrix $M_n(\Omega)$, introduced in in  Lemma \ref{lemma linearized operator at 0} and defining the linearized operator $d_r\mathcal{F}(\Omega,0)$, and, second, to establish the remaining prerequisite properties of that operator before we apply Crandall-Rabinowitz’s theorem.

  \begin{prop}\label{lem-spec} 
  For a given $n\in \mathbb{N}^*,$ the matrix $M_n(\Omega)$, introduced in  Lemma \ref{lemma linearized operator at 0} above, is not invertible if and only if $\Omega = \Omega_n^{\pm},$
where 
\begin{equation}\label{Omeg+-}
\Omega_n^{\pm}\bydef \frac{1}{2(\delta + 1)} \left(  - (A_n+B_n)\pm  \sqrt{(A_n-B_n)^2 +{4\delta}\left(\frac{b^{n}}{2n}-I_{n}(b_2\mu )K_{n}(b_1\mu)\right)^2} \right),
\end{equation}
where, we set 
\begin{equation}\label{def AnBn}
     \begin{aligned}
	A_{n}&\bydef   (\delta+1)V_{b_1,b_2}+\frac{\delta}{2n}+I_{n}(b_1\mu)K_{n}(b_1\mu) 
	\\
	B_n&\bydef   (\delta+1)W_{b_1,b_2}+\frac{1}{2n}+{\delta} I_{n}(b_2\mu)K_{n}(b_2\mu),
\end{aligned}
\end{equation} 
  and we recall that $V_{b_1,b_2}$ and $W_{b_1,b_2}$ are  defined in \eqref{V120}, above.  
  Moreover, for any $ 0<b_2\leq b_1$ and $\delta>0$, the sequences $(\Omega_n^\pm)_{n\in \mathbb{N}}$ are strictly increasing. In addition to that, there is $p_0\in \mathbb{N}^*$ such that, for all $m,n\geq  p_0$, any $0<b_2<b_1$ and   $\delta\geq \big (\frac{b_2}{b_1}\big)^2$, it holds that 
  $$ \Omega_n^+ > \Omega_m^-.$$

    Furthermore, regarding $\Omega_{n}^\pm$ as functions of $b_2\in (0,b_1)$, for fixed $b_1\in (0,\infty)$ and $n\in \mathbb{N}^*$, and introducing the set 
    $$ {S}_{m,b_1}\bydef    \left\{ b_2\in (0,b_1): \exists n\in\mathbb{N}^*,\ \Omega_{m}^-(b_2) =   \Omega_n^+(b_2) \right\},$$
    it then holds, for any $m\in \mathbb{N}^*$, that the set ${S}_{m,b_1}$ contains at most a finite number of elements.
    
    At last, in the particular case $b_1=b_2$, there exists  $b_1 \in (0,\infty)$  for which one has that
     \begin{equation*}
  	\Omega_n^+=  \Omega_m^-, \quad \text {for some} \quad n,m\in \mathbb{N}^*.
  \end{equation*}     
    \end{prop}
      \begin{rem}
      	It is readily seen that the set ${S}_{m,b_1}$ is empty, for any $m>p_0$, where $p_0$ is as per given in the statement of Proposition \ref{lem-spec}, above. However, the arguments in our proof below are not enough to show that this set is empty for low $m$-fold symmetries. This remains then unclear for now and can probably be wrong in view to the example of spectral collisions given in the statement the same lemma above in the case $b_1=b_2$.
      \end{rem} 
  \begin{proof}
  We split the proof into four steps for a better readability.  
  \subsubsection*{Computing the eigenvalues}
  We  first proceed  by noticing  that the matrix introduced in Lemma \ref{lemma linearized operator at 0} can be recast as 
 \begin{equation*} 
	 M_{n}(\Omega)=   \begin{pmatrix}
		\Omega+\frac{A_n}{\delta+1}
		& \frac{1}{\delta+1}\Big(\frac{b^{n}}{2n}-I_{n}(b_2\mu)K_{n}(b_1\mu)\Big)\vspace{5mm}\\
		\frac{\delta}{\delta+1}\Big(\frac{ b^{n}}{2n}-  I_{n}(b_2\mu  )K_{n}(b_1\mu)\Big)
		& \Omega+\frac{B_n}{\delta+1},
	\end{pmatrix},
	\end{equation*}
	for any $n\in \mathbb{N}^*$, where $A_n$ and $B_n$ are defined in the statement of Proposition \ref{lem-spec}, above.
	Accordingly,   it is readily seen that 
	\begin{equation*}
		\begin{aligned}
			\textnormal{det} M_{n}&=\left(\Omega+\frac{A_n}{\delta+1}\right)\left(\Omega+\frac{B_n}{\delta+1}\right) -\frac{\delta}{(\delta+1)^2}\left(\frac{b^{n}}{2n}-I_{n}(b_2\mu )K_{n}(b_1\mu)\right) ^2
	\\
	 &= \Omega^2+\frac{A_n+B_n}{\delta+1}\Omega+\frac{A_n B_n}{(\delta+1)^2}-\frac{\delta}{(\delta+1)^2}\left(\frac{b^{n}}{2n}-I_{n}(b_2\mu )K_{n}(b_1\mu)\right)^2.
		\end{aligned}
	\end{equation*} 
    Therefore, noting that the  discriminant of the preceding second order polynomial  satisfies
	\begin{equation*}\label{deltan}
		\begin{aligned}
			(\delta+1)^2\Delta_{n}
			&=\big(A_n+B_n\big)^2- 4A_nB_n +{4\delta}\left(\frac{b^{n}}{2n}-I_{n}(b_2\mu )K_{n}(b_1\mu)\right)^2
			\\ 
			&=\big(A_n-B_n\big)^2 +{4\delta}\left(\frac{b^{n}}{2n}-I_{n}(b_2\mu )K_{n}(b_1\mu)\right)^2> 0 ,
		\end{aligned}
	\end{equation*} 
it then follows, for all $n\geq 1$, that  there exist two angular velocities  $\Omega_n^{\pm}$ which are as given they are introduced in the lemma, and for which the matrix $M_n(\Omega_n^{\pm})$ is singular.   
 \subsubsection*{Monotony of the sequences $(\Omega_n^\pm)_{n\in \mathbb{N}^*}$} 
 Now, we provide a more precise analysis of the angular velocities $ \Omega_n^{\pm}$. To that end, note first, due to the representations    
\begin{equation*}
	A_n  
 	= -\delta \left(  \frac{1}{2} - \frac{1}{2n}\right) - b \left( \frac{b}{2} - I_1(b_2 \mu)K_1(b_1 \mu)  \right) - \Big (I_1(b_1 \mu)K_1(b_1 \mu)  - I_n(b_1 \mu)K_n(b_1 \mu)  \Big)
\end{equation*}
and 
\begin{equation*}
	B_n = -  \left(  \frac{1}{2} - \frac{1}{2n}\right) - \frac{\delta}{b} \left( \frac{b}{2} - I_1(b_2 \mu)K_1(b_1 \mu)  \right) - \delta\Big (I_1(b_2 \mu)K_1(b_2 \mu)  - I_n(b_2 \mu)K_n(b_2 \mu)  \Big),
\end{equation*} 
 together with Lemma \ref{lemma InKn}, that the sequences $(A_n)_{n\in \mathbb{N}^*}$ and $(B_n)_{n\in \mathbb{N}^*}$ are non-positive and non-increasing.  
 
  Now, we claim that $(\Omega_{n}^\pm)_{n\in \mathbb{N}^*}$, defined in \eqref{Omeg+-}, are both non-negative, non-decreasing sequences. To see that, we regard these sequences as functions of the real variable $x\in [1,\infty)$, i.e., we now consider the functions $x\mapsto \Omega_{x}^\pm $ which extend   the preceding sequences on   $[1,\infty)$. Note that it is easy to check that these functions are differentiable. Therefore, denoting 
 $$
M_x\bydef   \frac{A_x-B_x}{\delta+1},\qquad \Gamma_x\bydef   \frac{\frac{ b^{x}}{2x}- I_{x}(b_2\mu)K_{x}(b_1\mu)}{\delta+1},
$$
 we compute, for any $x\in [1,\infty)$, that  
  \begin{equation*}
 	\begin{aligned}
 	2(\delta+1)\partial_x\Omega_x^{\pm} 
 		&=\left( \pm\frac{M_x}{\sqrt{M_x^2+4\delta \Gamma_x^2}}-1\right)\partial_x A_x -\left(\pm\frac{M_x}{\sqrt{M_x^2+4\delta \Gamma_x^2}}+1\right) \partial_x B_x 
 		 \\
 		& \qquad \pm \frac{4\delta \Gamma_x}{\sqrt{M_x^2+4\delta \Gamma_x^2}}\partial_x\left(\frac{ b^{x}}{2x}- I_{x}(b_2\mu)K_{x}(b_1\mu)\right).
 	\end{aligned}
 \end{equation*} 
 Then, from the expression of $A_x$ and $B_x$, given by \eqref{def AnBn}, we get
 \begin{equation*}
 	\begin{aligned}
 	2(\delta+1)\partial_x\Omega_x^{\pm}      
        &=-\left(\delta+1\mp\frac{(\delta-1) M_x}{\sqrt{M_x^2+4\delta \Gamma_x^2}}\right)\partial_x  \left(\frac{1}{2x}\right)
        \\
        &\qquad  -\left(1\mp\frac{M_x}{\sqrt{M_x^2+4\delta \Gamma_x^2}}\right)\partial_x  \Big ( I_{x}(b_1\mu)K_{x}(b_1\mu)\Big) 
         \\ 
        &\qquad -\delta\left(1\pm\frac{M_x}{\sqrt{M_x^2+4\delta \Gamma_x^2}}\right) \partial_x  \left( I_{x}(b_2\mu)K_{x}(b_2\mu)\right)
        \\
        & \qquad \pm \frac{4\delta \Gamma_x}{\sqrt{M_x^2+4\delta \Gamma_x^2}}\partial_x\left(\frac{ b^{x}}{2x}- I_{x}(b_2\mu)K_{x}(b_1\mu)\right).
 	\end{aligned}
 \end{equation*}
 Hence, one deduces that 
 \begin{equation*}
 	\begin{aligned}
 		2(\delta+1)\partial_x\Omega_x^{-}
         &=-(\delta+1)\Bigg(\frac{\sqrt{M_x^2+4\delta \Gamma_x^2}+ M_x}{\sqrt{M_x^2+4\delta \Gamma_x^2}}\Bigg)\partial_x  \left(\frac{1}{2x}\right)
          \\
          & \quad  -\Bigg(1+\frac{M_x}{\sqrt{M_x^2+4\delta \Gamma_x^2}}\Bigg)\partial_x  \Big( I_{x}(b_1\mu)K_{x}(b_1\mu)\Big)  
      \\ &\quad -\delta\Bigg(1-\frac{M_x}{\sqrt{M_x^2+4\delta \Gamma_x^2}}\Bigg)\partial_x  \Big( I_{x}(b_2\mu)K_{x}(b_2\mu)\Big)
      \\
      & \quad - \frac{4\delta \Gamma_x}{\sqrt{M_x^2+4\delta \Gamma_x^2}}\partial_x\Big(\frac{ b^{x}}{2x}- I_{x}(b_2\mu)K_{x}(b_1\mu)\Big) .
 	\end{aligned}
 \end{equation*} 
Likewise for $\Omega_x^+$, by further employing the straitforward computation 
\begin{equation*}
	\begin{aligned}
		\delta+1-\frac{(\delta-1) M_x+4\delta \Gamma_x}{\sqrt{M_x^2+4\delta \Gamma_x^2}}&=\frac{(\delta+1)\sqrt{M_x^2+4\delta \Gamma_x^2}-(\delta-1) M_x- 4\delta \Gamma_x}{\sqrt{M_x^2+4\delta \Gamma_x^2}}
		\\ 
		&=\frac{(\delta+1)^2({M_x^2+4\delta\Gamma_x^2})-\big((\delta-1)M_x+4\delta\Gamma_x\big)^2}{\sqrt{M_x^2+4\delta\Gamma_x^2}\big((\delta+1)\sqrt{M_x^2+4\delta\Gamma_x^2}+(\delta-1)M_x+ 4\delta\Gamma_x\big)}
		\\ 
		&=4\delta\left(\frac{ M_x^2+(\delta-1)^2\Gamma_x^2-2(\delta-1)M_x\Gamma_x}{\sqrt{M_x^2+4\delta\Gamma_x^2}\big((\delta+1)\sqrt{M_x^2+4\delta\Gamma_x^2}+(\delta-1)M_x+4\delta\Gamma_x\big)}\right)
		\\ 
		&=\frac{4\delta \big(M_x-(\delta-1)\Gamma_x\big)^2}{\sqrt{M_x^2+4\delta\Gamma_x^2}\big((\delta+1)\sqrt{M_x^2+4\delta\Gamma_x^2}+(\delta-1)M_x+4\delta\Gamma_x\big)} ,
	\end{aligned}
\end{equation*}
it is then readily seen that 
\begin{equation*}
	\begin{aligned}
		 2(\delta+1)\partial_x\Omega_x^{+}
		 &=-\bigg(\frac{4\delta \big(M_x-(\delta-1)\Gamma_x\big)^2}{\sqrt{M_x^2+4\delta\Gamma_x^2}\big((\delta+1)\sqrt{M_x^2+4\delta\Gamma_x^2}+(\delta-1)M_x+4\delta\Gamma_x\big)}\bigg)\partial_x  \left(\frac{1}{2x}\right)
		 \\ 
		 &\quad- \frac{4\delta \Gamma_x}{\sqrt{M_x^2+4\delta \Gamma_x^2}}\partial_x\left( \frac{1}{2x}-\frac{ b^{x}}{2x} \right) 
		 \\
		 &\quad-\bigg(1-\frac{M_x}{\sqrt{M_x^2+4\delta \Gamma_x^2}}\bigg) \partial_x  \Big( I_{x}(b_1\mu)K_{x}(b_1\mu)\Big) 
		 \\
		 &\quad -\delta\bigg(1+\frac{M_x}{\sqrt{M_x^2+4\delta \Gamma_x^2}}\bigg)\partial_x  \Big(  I_{x}(b_2\mu)K_{x}(b_2\mu)\Big) 
		 \\
		 &\quad- \frac{4\delta \Gamma_x}{\sqrt{M_x^2+4\delta \Gamma_x^2}}\partial_x\Big(  I_{x}(b_2\mu)K_{x}(b_1\mu)\Big) .
	\end{aligned}
\end{equation*}
At last, the conclusion of the proof of the monotonicity if  the function $ x\mapsto \Omega_x^\pm$ follows from the fact that the functions 
\begin{equation*}
	x \mapsto \frac{1}{x}, \qquad x\mapsto  I_{x}(b_k\mu)K_{x}(b_j\mu), \qquad  x\mapsto \frac{1}{2x}-\frac{ b^{x}}{2x}
\end{equation*}
are decreasing, for any $k,j\in \{1,2\}$ as soon as $b_k \leq b_j$, and $b\in (0,1)$, altogether with the fact that $\Gamma_x \geq 0$ and that
\begin{equation*} 
 -\sqrt{M_x^2+4\delta\Gamma_x^2}\leqslant M_x\leqslant \sqrt{M_x^2+4\delta\Gamma_x^2},
\end{equation*}
for any $x\in [1,\infty)$. 
As a consequence of the preceding computations, we deduce that
\begin{equation*}
	\Omega_n^\pm \neq \Omega_m^\pm,
\end{equation*}
for all $m\neq n \in \mathbb{N}^*$. We will now precise sufficient conditions that will ensure that 
 \begin{equation*}
	\Omega_n^+ \neq \Omega_m^-.
\end{equation*}
To that end, we first observe that 
\begin{equation*}
	\Omega_n^- < \Omega_n^+,
\end{equation*}
for all $n\in \mathbb{N}^*$. Moreover, by denoting
\begin{equation*}
	A_\infty \bydef \lim_{n\to \infty} A_n \quad \text{and} \quad B_\infty \bydef \lim_{n\to \infty} B_n,
\end{equation*}
it then follows that  that 
\begin{equation*}
	\Omega_\infty^\pm \bydef  -\frac{1}{2(\delta+1) } \Big  (  A_\infty + B_\infty  \mp | A_\infty - B_\infty| \Big ).
\end{equation*}
On the other hand,   we observe that 
\begin{equation*}
	A_\infty = - \left( \frac{b^2+\delta}{2} +  I_1(b_1 \mu)K_1(b_1 \mu) - b    I_1(b_2 \mu)K_1(b_1 \mu) \right),
\end{equation*}
whereas
\begin{equation*}
B_\infty = - \left( \frac{ 1+ \delta}{2} + \delta I_1(b_2 \mu)K_1(b_2 \mu) - \frac{\delta}{b}    I_1(b_2 \mu)K_1(b_1 \mu) \right).
\end{equation*}
Therefore, one sees that 
\begin{equation}\label{A-B}
	\begin{aligned}
		A_\infty - B_\infty
		& =  \frac{1-b^2}{2}  - I_1(b_1 \mu)K_1(b_1 \mu) + \delta I_1(b_2 \mu)K_1(b_2 \mu)   
		\\
		&+  \left(    b^2 - \delta \right)  \frac{1}{b} I_1(b_2 \mu)K_1(b_1 \mu).
	\end{aligned}
\end{equation}
Observe that the value of preceding quantity is identically zero when $b_1=b_2$. However, we are going to show now that it is strictly positive elsewhere. To that end, we are going to employ the inequalities 
\begin{equation*} 
	 I_1(b_1\mu) >\frac{1}{b} I_1(b_2\mu) 
\end{equation*} 
and
\begin{equation*} 
	I_1(b_2\mu)K_1(b_2\mu) > I_1(b_1\mu)K_1(b_1\mu)
\end{equation*}
which are valid for any $b_2< b_1$ and follow directly from Lemma  \ref{lemma InKn}. Now, observe that, under   the assumption $ \delta\geq b^2$, the preceding bounds yields that 
\begin{equation*}
	\begin{aligned}
		A_\infty - B_\infty 
	& >   \frac{1-b^2}{2} +(\delta -1 )  I_1(b_1 \mu)K_1(b_1 \mu)    +  \left(    b^2 - \delta \right)  \frac{1}{b} I_1(b_2 \mu)K_1(b_1 \mu)
	\\
	& >   \frac{1-b^2}{2} +(\delta -1 )  I_1(b_1 \mu)K_1(b_1 \mu)    +  \left(    b^2 - \delta \right)  I_1(b_2 \mu)K_1(b_1 \mu)
	\\
	& =   (1-b^2)  \left( \frac{1}{2} -   I_1(b_1 \mu)K_1(b_1 \mu)   \right) >0,
	\end{aligned}
\end{equation*}
where we have utilized the second point from Lemma \ref{lemma InKn} in the last line. 
Accordingly, due to the monotonicity of the sequences $(\Omega_n^\pm)_{n\in\mathbb{N}^*}$, it is then readily seen that one can find $p_0\in \mathbb{N}^*$ such that 
$$ \Omega_m^+\geq  \Omega_{p_0}^+ >  \Omega_\infty^- > \Omega_n^-,$$
for all $m\geq p_0 $ and $n\in \mathbb{N}^*$, thereby concluding that 
\begin{equation*}
	  \Omega_m^+ > \Omega_n^-, 
\end{equation*}
for all $n,m\geq p_0.$

\subsubsection*{Employing analyticity to study the cardinal of spectral collisions}

Let us  now assume that $ 1\leq  n < p_0$ and fix $m\in \mathbb{N}^*$ with $1\leq n < m$.  Also, we now regard $\Omega_n^{\pm}$ as    analytic functions of  the real variable $b_2\in (0,b_1)$. We  emphasize that this is a consequence of the analytic property of the Bessel functions defining the eigenvalues $\Omega_n^{\pm}$, see \cite[Appendix B.2]{MF10}.  
  
  Thus, by virtue of the monotony of the eigenvalues $n\mapsto \Omega_{n}^\pm $, we infer, for any $ 1\leq  n \leq p_0$, that there is at most one index $m>n$   for which the equality
$$\Omega_m^-  =   \Omega_n^+  $$
can probably hold for some $b_2\in (0,b_1)$.  Accordingly, for $1\leq n \leq p_0$, we introduce the (possibly empty) set 
  $$ S_{n,b_1}\bydef    \left\{ b_2\in (0,b_1): \exists m=m(n)>n, \quad  \text{such that} \quad  \Omega_{m(n)}^-(b_2) =   \Omega_n^+(b_2)  \right\}    .$$ 
  Therefore, we claim that there is      $q_0\in \mathbb{N}^*$ and a finite sequence of real numbers $(c_j)_{1\leq j \leq q_0}$ such that 
\begin{equation}\label{finite:elements}
\bigcup_{ 1\leq n \leq  p_0 } S_{n,b_1} \subset \left\{ c_{j} \in (0,1):  1\leq j \leq q_0   \right\}.
\end{equation}
  
  In other words, loosely speaking, we  claim that the set of   values  $b_2\in (0,b_1)$ for which the sequence of eigenvalues of the matrix $ M_n(\Omega)$ can match is, in worse cases, negligible for Lebesgue's measure.   Again, showing that relies on the analytic property of   the function 
  $$b_2\mapsto \Omega_n^{+}(b_2) - \Omega_{m(n)}^{-}(b_2) $$ 
  on $(0,b_1)$, for all $n\in \mathbb{N}^*,$ which yields that this non constant function   has at most isolated zeros on open connected sets.  
   Consequently,  we deduce that 
  $$S_{n,b_1}  \subset  \left\{ c_{j,n} \in (0,b_1):   1\leq   j \leq j_*    \right\}, \quad \text{for some} \quad j_* \in \mathbb{N}^*,$$ 
   and some real numbers $c_{j,n} \in (0,b_1)$, whereby \eqref{finite:elements} follows.  
   At last, we deduce    that 
   $$ \Omega_n^+  \neq \Omega_m^- ,  $$
   for any $b_2\in (0,b_1)\backslash \{ c_j\}_{1\leq j\leq q_0}$ and all $  n,m\in \mathbb{N}^*$.

\subsubsection*{Example of spectral collisions} Let us now   consider the case of identical initial discs, i.e., $b_1=b_2$ and show that, in this particular situation, there exist values of $b_1\in (0,\infty)$ for which the sequence of matrix $\big (M_n(\Omega)\big)_{n\in \mathbb{N}^*}$ possess a simple sequence of eigenvalues, i.e., that 
  \begin{equation*}
  	\Omega_n^+\neq \Omega_n^-, \quad \text {for all} \quad n \in \mathbb{N}^*,
  \end{equation*}   
 but with possible ``counterpart collision'', i.e, 
  \begin{equation*}
  	\Omega_n^+=  \Omega_m^-, \quad \text {for some} \quad n,m\in \mathbb{N}^*.
  \end{equation*}  
  
To see that, we emphasize that a direct computation of the angular velocities $\Omega_n^\pm$, by setting $b_1=b_2$ in \eqref{Omeg+-}, yields that 
\begin{equation}\label{EV:delta=1}
	\Omega_n^+ = \frac{1}{2} - I_n K_n(b_1\mu)  \qquad \text{and} \qquad \Omega_n^- = \frac{1}{2} - \frac{1}{2n}.
\end{equation}
Therefore, in view of Lemma \ref{lemma InKn} above, it is easy to see that both sequences $(\Omega_n^\pm)_{n\in \mathbb{N}^*}$ are increasing and converge to the same limit $\frac{1}{2}$. Also, it is readily seen, again by virtue of Lemma \ref{lemma InKn}, that 
\begin{equation*}
	\Omega_n^+ > \Omega_n^-, \quad \text{for all} \quad n \in \mathbb{N}^*,
\end{equation*}
and, whence, by monotony of the sequence $ (\Omega_n^-)_{n\in \mathbb{N}^*}$, that 
\begin{equation*}
		\Omega_n^+\neq \Omega_m^-, \quad \text {for all} \quad n\geq m\in \mathbb{N}^*.
\end{equation*}
However,   using the asymptotic properties of the function $x\mapsto I_nK_n(x)$ to write that
\begin{equation*}
	\lim_{x\to 0} I_1K_1(x) = \frac{1}{2} \qquad \text{and} \qquad \lim_{x\to \infty} I_1K_1(x) = 0,
\end{equation*}
meaning that $x\mapsto  I_1K_1 (x)$ is a continuous function and takes values in $(0,\frac{1}{2})$, allows us to deduce the existence of $x_0\in (0,\infty)$ and $m\in \mathbb{N}^*$ such that 
\begin{equation*}
	I_1K_1 (x_0) = \frac{1}{2m},
\end{equation*}
 thereby deducing that there is $b_1=b_2\in (0,\infty)$ such that 
 \begin{equation*}
 	\Omega_1^+ = \Omega_m^-, \quad \text{for some} \quad m \in \mathbb{N}^*.
 \end{equation*}
 This shows the existence of spectral collisions in the case $b_1=b_2$ and concludes the proof of the proposition. 
   \end{proof} 

\subsection{Applying Crandall-Rabinowitz's theorem and proof of Theorem \ref{Thm:2-A}}
 	   Now, we are in position to establish the last prerequisites before we apply   Crandall-Rabinowitz’s theorem. In particular, in the following proposition, we   prove essential properties of the kernel and image of the linearized operator $d_r \mathcal{F}$ along with the transversality condition.

  \begin{prop}\label{prop.last}
Let $m\in \mathbb{N}^*$ and $\alpha\in (0,1)$ be fixed and assume that the radii of the initial discs are such that $0<b_2< b_1$. Set $b=\frac{b_2}{b_1}$, further assume that $\delta\geq b^2$  and  either $m\geq p_0$ or  $b_2\notin S_{m,b_1}$, where $p_0$ and $S_{m,b_1}$ are introduced in Proposition \ref{lem-spec} above.  Then, the following holds true:
\begin{enumerate}
\item The linearized  operator $d_r \mathcal{F}(\Omega,0)$ has a 
non trivial kernel in $\mathcal{X} _m^\alpha$ if and only if $\Omega=\Omega^{\pm}_{m n}$ for some $n\in\mathbb{N}^*$, where $\Omega^{\pm}_{m n}$ is given by \eqref{Omeg+-}. In this case, the kernel of $d_r \mathcal{F}(\Omega_m^{\pm},0)$ is a one-dimensional vector space in $\mathcal{X}_m^\alpha $ generated by  
\begin{equation*}
	\theta \mapsto 
	\begin{pmatrix}
		  \Omega_m^\pm + \frac{B_m}{\delta + 1 }
		  \vspace{3mm}\\ 
		  -\frac{\delta}{\delta+1}  \gamma_m 	
		  \end{pmatrix}
		   \cos(m \theta ) ,
\end{equation*} 
where, $B_m$ is defined by \eqref{def AnBn} and we set
\begin{equation}\label{sigma:def}
  \gamma_m \bydef    \frac{b^{m}}{2m}-I_{m}(b_2\mu )K_{m}(b_1\mu), \qquad b= \frac{b_2}{b_1} .
\end{equation} 
\item  The range of $d_r \mathcal{F}(\Omega^{\pm}_{m },0)$ is closed in $\mathcal{Y}_m^\alpha$ and is of co-dimension one.

\item At last, the transversality condition holds, i.e., we have that
  $$\partial_\Omega d_r \mathcal{F}(\Omega^{\pm}_{m },0)  \begin{pmatrix}
		1\\
		\sigma_m
	\end{pmatrix}    \cos( m \cdot)  \notin R\big(\partial_r \mathcal{F}(\Omega^{\pm}_{m},0)\big).$$
\end{enumerate}
\end{prop}

\begin{proof}
Let $\rho=(\rho_1,\rho_2)\in \mathcal{X} _m^\alpha$ with  Fourier series  expansions
$$   \rho_{k}(\theta)=\displaystyle\sum_{n\geq 1}c_{nm,k}\cos(n m\theta) , $$
for some $c_{n,k}\in\mathbb{R}$ and all $ k\in\{1,2\}$.
 According to Lemma \ref{lemma linearized operator at 0} above, one has that  
	\begin{equation}\label{def MjbO2}
	d_r\mathcal{F}(\Omega,0)\begin{pmatrix}
		{\rho}_{1}\vspace{0.1cm}\\
		{\rho}_{2}
	\end{pmatrix}=-\sum_{n\geq 1}n m M_{n m}(\Omega)\begin{pmatrix}
		c_{n m,1}\vspace{0.1cm}\\
		c_{n m,2}
	\end{pmatrix}\sin(n m\theta).
	\end{equation}
	Hence, in view of Proposition \ref{lem-spec}, the determinant of the matrix $M_{m}(\Omega)$ vanishes if and only if $\Omega=\Omega^{\pm}_{m}$.
	Thus, the kernel of $d_r\mathcal{F}(\Omega^{\pm}_{m},0)$ is non trivial and it is one-dimensional if and only if for all $n\geq 2$,
	\begin{equation}\label{det=/=0}
	\textnormal{det}\big(M_{n m}(\Omega^{\pm}_{m })\big)\neq 0 
	\end{equation}
which is equivalent to
$$
\Omega^{\pm}_{m }\neq\Omega^{\pm}_{n m } \quad \textnormal{for all}\quad n\neq 1.
$$
The preceding condition is ensured by Proposition \ref{lem-spec} whenever one the two following conditions are fulfilled: either $m >p_0$ or that $b_2\notin S_{m,b_1}$.  
Now, observe that  $\rho=(\rho_1,\rho_2)$ belongs to the
kernel of $d_r\mathcal{F}(\Omega^{\pm}_{m},0)$ if and only the Fourier coefficients in its Fourier expansion  \eqref{def MjbO2} vanish, i.e, if
$$
c_{nm,1}=c_{nm,2}=0,   $$
for all $ n\neq 1$, and 
$$ (c_{m,1},c_{m,2})\in\textnormal{ker}\big(M_{m}(\Omega^{\pm}_{m})\big).
$$ 
Hence, by noticing that 
\begin{equation*}
	 \begin{pmatrix}
		\Omega_m^\pm+\tfrac{A_n}{\delta+1}
		& \tfrac{\gamma_m }{\delta+1}\vspace{5mm}\\
		\tfrac{\delta}{\delta+1}\gamma_m 
		& \Omega_m^\pm+\tfrac{B_n}{\delta+1},
	\end{pmatrix} \cdot 
	\begin{pmatrix}
		  \Omega_m^\pm + \frac{B_m}{\delta + 1 }
		  \vspace{3mm}\\ 
		  -\frac{\delta}{\delta+1}  \gamma_m 	
		  \end{pmatrix} 
		  =
		  \begin{pmatrix}
		  0
		   \vspace{2mm}\\
		  0	
		  \end{pmatrix} ,
\end{equation*}
 one then deduces that the generator of $\textnormal{ker}\big(d_r\mathcal{F}(\Omega^{\pm}_{m},0)\big)$ can be chosen as the pair of functions
$$
\theta \mapsto \begin{pmatrix}
		  \Omega_m^\pm + \frac{B_m}{\delta + 1 }
		  \vspace{3mm}\\ 
		  -\frac{\delta}{\delta+1}  \gamma_m 	
		  \end{pmatrix} \cos( m \theta).
$$ 
This takes care of the first claim in the proposition.  
Now, we turn our attention to show that the range of $d_r\mathcal{F}(\Omega^{\pm}_{m},0)$ coincides with the space  of   functions $(g_1,g_2)\in \mathcal{Y}_m^\alpha$ such that
\begin{equation} \label{g:expression}
\begin{pmatrix}
		g_1\\
		g_2
	\end{pmatrix}=\sum_{n\geq 1} \begin{pmatrix}
		g_{nm,1}\\
		g_{nm,2}
	\end{pmatrix} \sin(n m \theta),
\end{equation}
where $g_{nm,1},g_{nm,2}\in\mathbb{R}$, for any $n\geq 2$, and there exists $c_{m,1},c_{m,2}\in\mathbb{R}$ such that
\begin{equation}\label{cond1}
M_m(\Omega_m^\pm)\begin{pmatrix}
		c_{ m,1}\vspace{0.1cm}\\
		c_{ m,2}
	\end{pmatrix}=\begin{pmatrix}
		g_{ m,1}\vspace{0.1cm}\\
		g_{ m,2}
	\end{pmatrix}.
\end{equation}

First, observe that the range of  $d_r\mathcal{F}(\Omega^{\pm}_{m},0)$ is obviously included in the space introduced above, which is clearly closed and of co-dimension one in $\mathcal{Y}_m^\alpha$.  Therefore, it only remains to check   the converse inclusion.  
To see that, fixing $(g_1,g_2)\in \mathcal{Y} _m^\alpha$, we shall prove that the equation 
\begin{equation} \label{rang:equa:1}
d_r\mathcal{F}(\Omega^{\pm}_{m},0)[(\rho_1,\rho_2)]=\begin{pmatrix}
		g_1\\
		g_2
	\end{pmatrix}
\end{equation} 
admits a solution in the space $\mathcal{X} _m^\alpha$, as soon as $(g_1,g_2) $   satisfies \eqref{g:expression} and \eqref{cond1} for some $c_{nm,1},c_{nm,2}\in\mathbb{R}$.
 To that end, we first notice that  \eqref{rang:equa:1} is equivalent to
$$
-n m M_{n m}(\Omega_m^{\pm})\begin{pmatrix}
		c_{n m,1}\vspace{0.1cm}\\
		c_{n m,2}
	\end{pmatrix}=\begin{pmatrix}
		g_{n m,1}\vspace{0.1cm}\\
		g_{n m,2}
	\end{pmatrix}, \quad \text{for all}\quad  n\geq 1.
$$
The existence of $c_{m,1}$ and $c_{m,2}$ is obviously ensured  by  \eqref{cond1}. As for the remaining     coefficients, it is readily seen that 
$$
\begin{pmatrix}
		c_{n m,1}\vspace{0.1cm}\\
		c_{n m,2}
	\end{pmatrix}=-\frac{1}{n m}M_{n m}^{-1}(\Omega_m^{\pm})\begin{pmatrix}
		g_{n m,1}\vspace{0.1cm}\\
		g_{n m,2}
	\end{pmatrix}, \quad \text{for all} \quad n\geq 2,
$$
where we have used the fact that the matrix $M_{n m} (\Omega_m^{\pm}) $  is invertible for all $n\geq 2$, which is established  at the beginning of this proof. This defines the coefficients $c_{nm,1}$ and $c_{nm,2}$.

Next, for a given   couple of functions $ (g_1,g_2) \in \mathcal{Y} _m^\alpha$ as in \eqref{g:expression},   we  need to   show   that the solution  of \eqref{rang:equa:1} belongs to the space $\mathcal{X} _m^\alpha$. More precisely, we claim that  
$$
\begin{pmatrix}
		\rho_1\\
		\rho_2
	\end{pmatrix}=  \begin{pmatrix}
		c_{m,1}\vspace{0.1cm}\\
		c_{  m,2}
	\end{pmatrix} \cos(  m \theta) -\sum_{n\geq 2} \frac{1}{n m}M_{n m}^{-1}(\Omega_m^{\pm})\begin{pmatrix}
		g_{n m,1}\vspace{0.1cm}\\
		g_{n m,2}
	\end{pmatrix} \cos(n m \theta)\in C^{1+\alpha}(\mathbb{T})\times C^{1+\alpha}(\mathbb{T}),
$$
where, for $c_{m,1}$ and $ c_{m,2} $ are defined in \eqref{cond1}. 
It is clear that we can restrict our selves to showing that 
$$\sum_{n\geq 2} \frac{1}{n m}M_{n m}^{-1}(\Omega_m^{\pm})\begin{pmatrix}
		g_{n m,1}\vspace{0.1cm}\\
		g_{n m,2}
	\end{pmatrix} \cos(n m \theta)\in C^{1+\alpha}(\mathbb{T})\times C^{1+\alpha}(\mathbb{T}),$$
	  which is equivalent to proving that 
	\begin{equation}\label{claim:g:Ca}
\begin{pmatrix}
		 \widetilde{\rho}_1\\
		\widetilde{\rho}_2
	\end{pmatrix} \bydef    \sum_{n\geq 2}  M_{n m}^{-1}(\Omega_m^{\pm})
\begin{pmatrix}
		g_{n m,1}\vspace{0.1cm}\\
		g_{n m,2}
	\end{pmatrix} \sin(n m \theta)\in C^{\alpha}(\mathbb{T})\times C^{\alpha}(\mathbb{T}).
	\end{equation}
Now, by a direct computation, it is easy to check that   
$$M_{n m}^{-1}(\Omega_m^{\pm}) = \frac{1}{\det(M_{n m} (\Omega_m^{\pm}))}  \begin{pmatrix}
	\Omega_m^{\pm} + \frac{B_{mn}}{\delta +1}	 & -  \frac{\gamma_{nm}}{\delta +1}   \vspace{0.1cm}\\
		  - \frac{ \delta \gamma_{nm}}{\delta +1}   & \Omega_m^{\pm} + \frac{A_{mn}}{\delta +1} 
	\end{pmatrix} ,\quad \text{for all} \quad n\geq 2  , $$
	where $\gamma_n$ is defined in \eqref{sigma:def} whereas $A_n$ and $B_n$ are introduced in \eqref{def AnBn}. 	
Note that the preceding matrix is well defined thanks to \eqref{det=/=0} which ensures that 	\begin{equation*}
	\begin{aligned}
	\det\big(M_{n m} (\Omega_m^{\pm})\big)  \neq 0 ,\quad \text{for all}  \quad n\geq 2 .
	\end{aligned}
	\end{equation*}
	We further introduce the matrix 
	$$P_{n,m} \bydef     M_{n m}^{-1}  (\Omega_m^{\pm})  -  M_{\infty}^{-1}  (\Omega_m^{\pm})   ,\quad \text{for all}  \quad n\geq 2 ,$$
	where 
	$$  M_{\infty}^{-1} (\Omega_m^{\pm}) \bydef    \lim_{n\rightarrow \infty}  M_{mn}^{-1}  (\Omega_m^{\pm}),$$
	and we claim that 
	\begin{equation}\label{rest:asymptotic}
	 |P_{n,m} | =O\left(\frac{1}{n}\right), \quad \text{for all} \quad n\geq 2.
\end{equation}	
	It is easy to see that the preceding claim follows as a consequence of   the following asymptotics:
	$$\left | \frac{1}{\det(M_{n m} (\Omega_m^{\pm})) } -  \frac{1}{\det(M_{\infty} (\Omega_m^{\pm})) } \right| =O\left(\frac{1}{n}\right) $$
	and 
	$$ \left |  \begin{pmatrix}
	  \frac{B_{mn}}{\delta +1}	 & -  \frac{\gamma_{nm}}{\delta +1}   \vspace{0.1cm}\\
		  - \frac{ \delta \gamma_{nm}}{\delta +1}   &   \frac{A_{mn}}{\delta +1} 
	\end{pmatrix} - \begin{pmatrix}
  \frac{B_{\infty}}{\delta +1}	 & 0 \vspace{0.1cm}\\
		 0 & \frac{A_{\infty}}{\delta +1} 
	\end{pmatrix} \right| =O\left(\frac{1}{n}\right)  ,$$
	as $n \to \infty$. These asymptotic identities  are a direct consequence of the decay properties of Bessel functions, see in particular \eqref{asymptotic:Bessel:1}. This justifies our claim \eqref{rest:asymptotic}. 
	
	Now, we are ready to conclude  the proof of our main claim \eqref{claim:g:Ca}. To that end, we first write that
	$$\begin{aligned}
	\begin{pmatrix}
		 \widetilde{\rho}_1\\
		\widetilde{\rho}_2
	\end{pmatrix} = \sum_{n\geq 2}  P_{n,m} (\Omega_m^{\pm})
\begin{pmatrix}
		g_{n m,1}\vspace{0.1cm}\\
		g_{n m,2}
	\end{pmatrix} \sin(n m \theta)  
	-  M_{\infty}^{-1}(\Omega_m^{\pm}) \sum_{n\geq 2}  
\begin{pmatrix}
		g_{n m,1}\vspace{0.1cm}\\
		g_{n m,2}
	\end{pmatrix} \sin(n m \theta) .
	\end{aligned}$$
	On the one hand, due to the assumption that $ (g_1,g_2) \in \mathcal{Y} _m^\alpha$, it is readily seen that 
	$$ M_{\infty}^{-1}(\Omega_m^{\pm}) \sum_{n\geq 2}  
\begin{pmatrix}
		g_{n m,1}\vspace{0.1cm}\\
		g_{n m,2}
	\end{pmatrix} \sin(n m \theta)  \in C^{ \alpha}(\mathbb{T})\times C^{ \alpha}(\mathbb{T}). $$
	On the other hand, since we have proved that 
	$$ P_{n,m} = O\left( \frac{1}{n}\right), \quad \text{for all} \quad n\geq 2,$$
	it then  follows that 
	$$ \theta\mapsto  \Theta(\theta) \bydef    \sum_{n\geq 2} P_{n,m}  \sin (nm\theta) \in  L^2(\mathbb{T}) \subset L^1(\mathbb{T}). $$ 
	Accordingly, we deduce that  
	$$\sum_{n\geq 2}  P_{n,m} (\Omega_m^{\pm})
\begin{pmatrix}
		g_{n m,1}\vspace{0.1cm}\\
		g_{n m,2}
	\end{pmatrix} \sin(n m \theta)   =  \Theta * \begin{pmatrix}
		g_{1}\vspace{0.1cm}\\
		g_{2}
	\end{pmatrix}    (\theta) \in C^{ \alpha}(\mathbb{T})\times C^{ \alpha}(\mathbb{T}), $$ 
	where the convolution above  is written in a vectorial format. This completes the proof of \eqref{claim:g:Ca}.
	
Finally, we are left with the proof of the transversality condition. To that end, we first  write, by differentiating \eqref{def MjbO2} with respect to $\Omega$, that
	\begin{equation*}
	\partial _\Omega d_r\mathcal{F}(\Omega,0)\begin{pmatrix}
		{\rho}_{1}\vspace{0.1cm}\\
		{\rho}_{2}
	\end{pmatrix}=-\sum_{n\geq 1}n \begin{pmatrix}
		c_{n,1}\vspace{0.1cm}\\
		c_{n,2}
	\end{pmatrix}\sin(n\theta),
	\end{equation*}
	for any $\Omega\in \mathbb{R}$.
Therefore, we obtain that
		\begin{equation*}
	\partial _\Omega d_r\mathcal{F}(\Omega_m^{\pm},0)
	\begin{pmatrix}
		  \Omega_m^\pm + \frac{B_m}{\delta + 1 }
		  \vspace{3mm}\\ 
		  -\frac{\delta}{\delta+1}  \gamma_m 	
		  \end{pmatrix}
		   \cos( m \theta)=- m
		   \begin{pmatrix}
		  \Omega_m^\pm + \frac{B_m}{\delta + 1 }
		  \vspace{3mm}\\ 
		  -\frac{\delta}{\delta+1}  \gamma_m 	
		  \end{pmatrix}\sin(m\theta). 
	\end{equation*}
	Hence, in view of \eqref{cond1},  it follows that 
	$$ \partial _\Omega d_r\mathcal{F}(\Omega_m^{\pm},0)\begin{pmatrix}
		  \Omega_m^\pm + \frac{B_m}{\delta + 1 }
		  \vspace{3mm}\\ 
		  -\frac{\delta}{\delta+1}  \gamma_m 	
		  \end{pmatrix} \cos( m \theta)\in {\rm R}\big(d_r\mathcal{F}(\Omega_m^{\pm},0)\big) $$
if and only if there exist  two real numbers $c_{m,1}$ and $c_{m,2}$ such that 
\begin{equation*} 
M_m(\Omega_m^\pm)\begin{pmatrix}
		c_{ m,1}\vspace{0.1cm}\\
		c_{ m,2}
	\end{pmatrix}=\begin{pmatrix}
		  \Omega_m^\pm + \frac{B_m}{\delta + 1 }
		  \vspace{3mm}\\ 
		  -\frac{\delta}{\delta+1}  \gamma_m 	
		  \end{pmatrix}.
\end{equation*}
Since $\Omega_m^\pm + \frac{B_m}{\delta + 1 } \neq 0 $, the latter identity is equivalent to finding two real numbers, still denoted by $c_{m,1}$ and $c_{m,2}$ for simplicity, such that 
\begin{equation}\label{transversality:pr}
M_m(\Omega_m^\pm)\begin{pmatrix}
		c_{ m,1}\vspace{0.1cm}\\
		c_{ m,2}
	\end{pmatrix}=\begin{pmatrix}
		  1\\ 
		  \sigma_m	
		  \end{pmatrix},
\end{equation}
where we set 
\begin{equation*}
	\sigma_m \bydef \frac{ -\frac{\delta}{\delta+1}  \gamma_m }{\Omega_m^\pm + \frac{B_m}{\delta + 1 }}
\end{equation*}
and we emphasize that  
\begin{equation*}
	\begin{pmatrix}
		  1\\ 
		  \sigma_m	
		  \end{pmatrix} \in \text{Ker} \big( M_m(\Omega_m^\pm) \big).
\end{equation*}
The validity of the   identity \eqref{transversality:pr}   is violated by means of basic linear algebra arguments, for $M_m(\Omega_m^\pm) $ is non-invertible and the vector $ (1,\sigma_m)$ is a non trivial zero of that matrix, with $\sigma_m \neq 0$. More precisely, this is a consequence of the following simple lemma
\begin{lem}\label{Lemma:ALG}
	Let $M$ be a real two-dimensional matrix with a non trivial kernel. Assume further that there is $s\in \mathbb{R}\setminus \{0\}$ such that 
	\begin{equation*}
		\begin{pmatrix}
		  1\\ 
		  s	
		  \end{pmatrix} \in \text{Ker} (M) \cap \text{R}(M).
	\end{equation*}
	Then, $M$ is trace-free, i.e., it holds that 
	\begin{equation*}
		\text{Tr} (M)=0.
	\end{equation*}
\end{lem}
Let us admit Lemma \ref{Lemma:ALG} for a moment and first continue  the proof of the transversality condition. We get back to the proof of that lemma thereafter.

According to Lemma \ref{Lemma:ALG}, in order to prove the non validity of \eqref{transversality:pr} for any real numbers $c_{m,1}, c_{m,2}$, it is enough for us to check that   $\text{Tr}(M_m(\Omega_m^\pm))\neq 0$.  To that end, we compute that 
		$$ \text{Tr}(M_m(\Omega_m^\pm)) = 2 \Omega_{m}^{\pm} + \frac{A_m+ B_m}{1+\delta}. $$
		Hence, in view  of the definition of $\Omega_m^{\pm}$, given in   \eqref{Omeg+-}, we find that 
		$$ \text{Tr}(M_m(\Omega_m^\pm)) = \pm (\Omega_m^+ - \Omega_m^-)  \neq 0,$$
		thereby deducing that \eqref{transversality:pr} cannot hold true  for any  real numbers $c_{m,1}$ and $c_{m,2}$. This justifies  the transversality condition.

Let us now prove Lemma \ref{Lemma:ALG}.  Generally speaking, a given two-dimensional matrix   $ M  $ is non-invertible if and only if 
$$M= \begin{pmatrix}
		a & \eta a \vspace{0.1cm}\\
		b & \eta b
	\end{pmatrix}, \quad \text{for some}\quad  a,b, \eta\in \mathbb{R} .$$
	Note that, without loss of generality,  we can assume that   $a$ and $b$ are both not trivial, otherwise we emphasize that there will be noting left to prove.	
	Therefore,  if $ (1,s)$ is in the kernel of $M$ and satisfies 
	$$M \begin{pmatrix}
		c_{ 1}\vspace{0.1cm}\\
		c_{  2}
	\end{pmatrix}=\begin{pmatrix}
		1\vspace{0.1cm}\\
		s	\end{pmatrix}, $$
		for some $s\in \mathbb{R}^*$ and some two real numbers $c_1$ and $c_2$, then one can easily check that  $M$ should have the  format 
		$$ M=  \begin{pmatrix}
		a &  - \frac{a^2}{b} \vspace{0.1cm}\\
		b &  -a 
	\end{pmatrix}.$$
		In particular, one deduces from all of this that $Tr(M)=0.$    This proves Lemma \ref{Lemma:ALG}  and concludes the proof of the proposition.  
\end{proof}

  \subsubsection*{Proof of Theorem \ref{Thm:2}}  The proof of Theorem 2.1 is now achieved as direct application of   Crandall--Rabinowitz's Theorem \ref{Crandall-Rabinowitz theorem}, whose hypothesis are fully fulfilled due to Proposition \ref{prop.last}, above.

  %\newpage
 \section{Endnote}
In this paper, it is shown that the multy-layer quasi--geostrophic system \eqref{EQ} admits a unique global weak---Lagrangian---solution as soon as the initial vorticities are Lebesgue-integrable and bounded. In light of that, any couple of initial vortex-patches is transported by the flow of the system and the vortex structure remains unchanged as time goes forward. Moreover, it is also shown here that there are time-periodic patches, close to stationary states (discs),    characterized by their rotational movement around their centers of mass with the same angular velocity.

Although the first result---existence of global weak solutions of \eqref{EQ}---is shown to hold in the full range of parameters $\lambda>0$ and $\delta>0$, the elements of proof of the second result---existence of rotating solutions---does not cover the whole previous range   of parameters. Specifically, technically speaking,  it turns out that the spectral analysis of the contour dynamics can be affected by the choice of these parameters and this is shown to also depend on the choice of the radii $b_1$ and $b_2$  of the initial discs.  Note that this phenomenon is not observable     in the Euler equations by virtue of their invariance by dilatation; it is, however, a property that comes from the contribution of the kernel associated with the shallow-water equations.

Our spectral analysis of the linearized operator around a steady state (two discs) shows that its kernel has a bi-dimensional structure. Moreover, the sequence of the associated eigenvalues $(\Omega_m^\pm)_{m\in \mathbb{N}^*}$ for which this operator is not invertible is, in some sense, a combination of a semblable sequence of   eigenvalues from Euler and shallow-water equations. This is well observed in the case $\delta=1$, see \eqref{EV:delta=1}. Although we are able to show that the linearized operator is not invertible at any value of angular velocities $ \Omega_m^\pm $, for any given symmetry $m\in \mathbb{N}^*$, it turns out that there are cases (depending on the choice of parameters $\lambda$, $\delta$, $b_1$ and $b_2$) where spectral collisions surely occur and the Crandall-Rabinowitz theorem does not directly apply to construct time-periodic solutions in such situations. Here, it is a remarkable observation  that the non-invertibility of the linearized operator at $ \Omega_m^\pm $, for any given symmetry $m\in \mathbb{N}^*$,   stems from the coupling of Euler and shallow-water kernels. This description for small number of symmetries is satisfactory when compared to the doubly connected case of shallow-water equations \cite{R21}, which, as previously emphasized, shares several aspects of similarity in terms of the two-dimensional structure of the V-states. 
 
In conclusion, the present work opens the gate for more interesting questions to be discussed about \eqref{EQ}, such as properties of the branches of bifurcation, description of stationary  solutions, existence of doubly connected V-states and   quasi-time-periodic solutions.

\section*{Acknowledgement} 
The authors would like to thank Taoufik Hmidi for introducing the multy-layer quasi--geostrophic system \eqref{EQ} as well as for several helpful discussions and remarks.

%We thank Emeric Roulley for pointing out a typo in a previous version of the manuscript which, when is corrected, enabled  us to slightly improve the statements of our results.

The work of Zineb Hassainia has been supported by Tamkeen under the NYU Abu Dhabi Research Institute grant of the center SITE.

% ==============
% = References =
% ==============

% \nocite*
\bibliographystyle{plain} 
\bibliography{BIB}
\end{document}